\documentclass[reqno,12pt]{amsart}

\usepackage{fullpage}

\usepackage{amssymb}
\usepackage{dsfont}

\usepackage[utf8]{inputenc} 
\usepackage[T1]{fontenc} 

\usepackage{hyperref}

\usepackage[nobysame,alphabetic,initials,msc-links]{amsrefs}

\DefineSimpleKey{bib}{how}

\renewcommand{\eprint}[1]{#1}
\BibSpec{misc}{%
  +{}{\PrintAuthors}  {author}
  +{,}{ \textit}      {title}
  +{,}{ }             {how}
  +{}{ \parenthesize} {date}
  +{,} { available at \eprint}        {eprint}
  +{,}{ available at \url}{url}
  +{,}{ }             {note}
  +{.}{}              {transition}
}

\numberwithin{equation}{section}

\theoremstyle{plain}
\newtheorem{thm}{Theorem}[section]
\newtheorem{prop}[thm]{Proposition}
\newtheorem{lemma}[thm]{Lemma}
\newtheorem{cor}[thm]{Corollary}
\theoremstyle{definition}

\newtheorem{defn}[thm]{Definition}
\theoremstyle{remark}

\newtheorem{remark}[thm]{Remark}
\newtheorem{example}[thm]{Example}

\newcommand\bp{\begin{proof}}
\newcommand\ep{\end{proof}}

\newcommand{\G}{{\mathcal{G}}}

\newcommand{\Gu}{{\mathcal{G}^{(0)}}}

\begin{document}

\title{Tracial weights on topological graph algebras}


\author{Johannes Christensen$^{1}$}
\thanks{$^{1}$KU Leuven, Department of Mathematics (Belgium).  E-mail:  johannes.christensen@kuleuven.be}

\thanks{Supported by a DFF-International Postdoctoral Grant.}

\begin{abstract}
We describe two kinds of regular invariant measures on the boundary path space $\partial E$ of a second countable topological graph $E$, which allows us to describe all extremal tracial weights on $C^{*}(E)$ which are not gauge-invariant. Using this description we prove that all tracial weights on the C$^{*}$-algebra $C^{*}(E)$ of a second countable topological graph $E$ are gauge-invariant when $E$ is free. This in particular implies that all tracial weights on $C^{*}(E)$ are gauge-invariant when $C^{*}(E)$ is simple and separable.
\end{abstract}

\maketitle

\section*{Introduction}
There is an abundance of C$^{*}$-algebras which can be realised as the C$^{*}$-algebras of topological graphs. To name a few examples, C$^{*}$-algebras arising from crossed products of $\mathbb{Z}$ by homeomorphisms, C$^{*}$-algebras arising from directed graphs, Kirchberg algebras satisfying the UCT and AF-algebras \cite{K2}. Since C$^{*}$-algebras of topological graphs provide a rich class of examples it is natural to study concepts like traces on such algebras.

The study of tracial states on topological graph C$^{*}$-algebras was initiated by Schaufhauser in \cite{S}. Schaufhauser proved that the gauge-invariant tracial states are affinely homeomorphic to a set of probability measures on the vertex set $E^{0}$, which we in this paper call the vertex-invariant measures. This generalises similar known results from the cases of directed graphs and crossed products by homeomorphisms, and it is known from these cases, that this simplification of the problem of describing gauge-invariant tracial states is very useful for providing concrete descriptions of traces. For directed graphs and crossed products it is also known that under the right assumptions, namely condition (K) for directed graphs and freeness of the action for crossed products, all tracial states are gauge-invariant. The notion of freeness of topological graphs generalises both of these conditions, which motivated Schaufhauser to make the following conjecture in \cite{S}. 

\bigskip

\noindent\emph{Conjecture: All tracial states on the C$^{*}$-algebra of a free topological graph are gauge-invariant}

\bigskip

\noindent If the C$^{*}$-algebra of a topological graph is simple, the graph has to be free, so a positive answer to this conjecture in particular implies that all tracial states on simple topological graph C$^{*}$-algebras can be described by vertex-invariant measures. This is a strong motivation for answering this conjecture, since it means that one can greatly simplify the problem of describing traces for simple topological graph C$^{*}$-algebras.  The last part of \cite{S} additionally argues that for many topological graphs the gauge-invariant tracial states are affinely homeomorphic with the states on the $K_{0}$-group of the graph C$^{*}$-algebra, which further stresses the importance of characterising which states are gauge-invariant.

\bigskip

The purpose of this paper is to continue the analysis of tracial states on topological graph C$^{*}$-algebras initiated in \cite{S}. In particular we will provide an affirmative answer to the above conjecture for separable C$^{*}$-algebras. In fact, we provide a graph-theoretical condition which we prove is equivalent to the existence of a non gauge-invariant tracial state c.f. Theorem \ref{thmnongauge} and Definition \ref{def217a}, and we then prove that a free graph can not satisfy this condition.

\bigskip

Instead of considering tracial states, we extend the scope of \cite{S} by considering \emph{tracial weights}. After  establishing the preliminaries, we therefore commence the paper by generalising the result of \cite{S} by establishing in Proposition \ref{propinvmeasurebij} a bijection between the following three sets:
\begin{itemize}
\item[-] The regular vertex-invariant measures on the vertex space $E^{0}$.
\item[-] The regular shift-invariant measures on the boundary path space $\partial E$ of $E$. 
\item[-] The gauge-invariant tracial weights on the graph C$^{*}$-algebra $C^{*}(E)$.
\end{itemize}
We then focus on two kinds of regular invariant measures on $\partial E$. First we analyse the extremal measures concentrated on the set of finite paths in $\partial E$, and in Theorem \ref{thmsingbi} we give a complete description of such measures. Secondly, we analyse the extremal invariant measures which are concentrated on paths in $\partial E$ which are eventually cyclic, and we completely describe these measures in Theorem \ref{thmcyclicbij}. Using this description, we provide both a necessary and sufficient condition for the existence of a non gauge-invariant \emph{tracial state} on $C^{*}(E)$, and a necessary and sufficient condition for the existence of a non gauge-invariant \emph{tracial weight}. Both of these conditions are graph-theoretical, which allows us prove that on a second countable free topological graph all tracial weights are gauge-invariant, proving that Schaufhauser conjecture is also true for tracial weights on separable $C^{*}$-algebras.


\section{Preliminaries}

\noindent We let $\mathbb{N}_{0}$ denote the natural numbers including zero.

\subsection{Topological graphs}
We will in the following introduce our terminology on topological graphs. We refer the reader to Katsura's original paper for more background \cite{K1}.

A topological graph $E=(E^{0}, E^{1}, r,s)$ consists of two locally compact Hausdorff spaces $E^{0}$ and $E^{1}$, a local homeomorphism $s: E^{1}\to E^{0}$ and a continuous map $r: E^{1} \to E^{0}$. We call $s$ the source map and $r$ the range map. Following the terminology for directed graphs, we refer to elements in $E^{0}$ as vertices and elements in $E^{1}$ as edges. We call the topological graph $E$ second countable when the two spaces $E^{0}$ and $E^{1}$ are second countable topological spaces.

We follow the convention for traversing paths which is used in \cite{K1, S}, i.e. a path $\alpha_{1}\alpha_{2} \cdots \alpha_{n}$ is a concatenation of edges $\alpha_{1}, \dots, \alpha_{n} \in E^{1}$ satisfying that $s(\alpha_{i})=r(\alpha_{i+1})$ for $i=1, \dots, n-1$. We say a path $\alpha:=\alpha_{1}\alpha_{2} \cdots \alpha_{n}$ has length $n$ and we write this $|\alpha|=n$. We interpret the vertices $E^{0}$ as paths of length zero. For any $n\in \mathbb{N}_{0}$ we denote by $E^{n}$ the paths of length $n$, which is then in keeping with the definition of vertices $E^{0}$ and edges $E^{1}$ for $n=0,1$. The space $E^{n}$ with $n\in \mathbb{N}$ inherits a locally compact Hausdorff topology by considering it as a subspace of the product space $\prod_{i=1}^{n} E^{1}$ with the product topology. Endowing $E^{n}$ with this topology for all $n\in \mathbb{N}$, we define the \emph{finite path space} $E^{*}$ as the set
$$
E^{*}:= \bigsqcup_{n=0}^{\infty} E^{n}
$$
with the topology generated by the topologies on the spaces $E^{n}$. We define the \emph{infinite path space} $E^{\infty}$ as the set of infinite concatenations $\alpha_{1}\alpha_{2}\cdots$ with $s(\alpha_{i})=r(\alpha_{i+1})$ for all $i\in \mathbb{N}$, and we write $|\alpha | =\infty$ for such a path $\alpha:=\alpha_{1} \alpha_{2} \cdots$. We can naturally extend the range map $r$ to a map on all paths in $E^{*}$ and $E^{\infty}$ by setting $r(\alpha)=r(\alpha_{1})$ when $|\alpha|\geq 1$ and $r(v)=v$ when $v\in E^{0}$. Likewise we can extend the source map to $E^{*}$ by setting $s(\alpha)=s(\alpha_{|\alpha|})$. The extensions of $r$ and $s$ to $E^{*}$ are then respectively also a continuous map and a local homeomorphism. We follow the standard convention of writing e.g. $AE^{*}B$ for paths $\alpha \in E^{*}$ with $r(\alpha) \in A$ and $s(\alpha) \in B$ when $A,B \subseteq E^{0}$.

There exists a maximal open subset $E_{reg}^{0}$ of $E^{0}$ satisfying that $r$ restricts to a proper surjection $r^{-1}(E_{reg}^{0}) \to E_{reg}^{0}$. More concretely $E_{reg}^{0}$ can be described as the set of vertices $v\in E^{0}$ satisfying that there exists an open neighbourhood $W$ of $v$ with $\overline{W}$ compact, $r^{-1}(\overline{W}) \subseteq E^{1}$ compact and $r(r^{-1}(W))=W$. We call $E^{0}_{reg}$ the set of \emph{regular vertices}, and we define $E^{0}_{sng}=E^{0} \setminus E^{0}_{reg}$ to be the set of \emph{singular vertices}. We set 
$$
E^{*}_{sng}:= E^{*} \cap s^{-1}(E^{0}_{sng}) \quad \text{ and } \quad
E^{n}_{sng}:= E^{n} \cap s^{-1}(E^{0}_{sng}) \quad \text{ for all $n\in \mathbb{N}$.}
$$

\subsection{The boundary path space of a topological graph}

We will in the following introduce the \emph{boundary path space $\partial E$} of a topological graph $E$. We refer the reader to \cite{Y} and Section 3 in \cite{S} for a more rigorous introduction.

As a set, we define the boundary path space $\partial E$ to be
$$
\partial E := E^{*}_{sng} \sqcup E^{\infty} \; .
$$ 
For any subset $S\subseteq E^{*}$ we define
$$
Z(S):= \{ \alpha \in \partial E\; | \; r(\alpha)\in S
\text{ or } \alpha_{1}\cdots \alpha_{n}\in S \text{ for some } 1\leq n \leq |\alpha| \} \; .
$$
Sets on the form $Z(U) \setminus Z(K)$, where $U\subseteq E^{*}$ is open and $K\subseteq E^{*}$ is compact, form a basis for a locally compact Hausdorff topology on $\partial E$. With this topology sets on the form $Z(K)$ are compact in $\partial E$ whenever $K\subseteq E^{*}$ is a compact set.

We define a map $\sigma:\partial E\setminus E^{0} \to \partial E$ by setting
$$
\sigma(\alpha_{1} \alpha_{2} \alpha_{3} \cdots ) 
=\alpha_{2} \alpha_{3} \cdots \quad \text{ for $\alpha=\alpha_{1} \alpha_{2} \alpha_{3} \cdots \in \partial E$ with $|\alpha|\geq 2$,} 
$$
and $\sigma(\alpha)=s(\alpha_{1})$ when $\alpha=\alpha_{1}\in \partial E$ satisfies $|\alpha |=1$. This map is a local homeomorphism which we call the \emph{backwards shift map}, c.f. Proposition 3.5 in \cite{S}. The observation in the following Lemma is implicitly used in the proof of Proposition 3.5 in \cite{S}. Since it will be vital for us at a later point, we will provide a proof for clarity. The statement of the Lemma also uses $\sigma$ to denote the map $E^{*} \setminus E^{0} \to E^{*}$ which is defined completely analogous to above. We also use the notation 
$$
UV=\{ \alpha \beta \; | \; s(\alpha)=r(\beta)\; ,\alpha \in U \text{ and } \beta \in V \} \subseteq E^{k+l}
$$
for subsets $U\subseteq E^{k}$ and $V \subseteq E^{l}$. Notice that $UV$ is open if $U,V$ are open, $UV$ is closed if $U,V$ are closed and $UV$ is compact if $U,V$ are compact.

\begin{lemma} \label{lemmanotE}
The set $\partial E \setminus E^{0}$ is open and has a basis of sets on the form $Z(U) \setminus Z(K)$ satisfying:
\begin{enumerate}
\item\label{item1E} $U\subseteq E^{*}$ is open and $K\subseteq E^{*}$ is compact , 
\item\label{item2E} $U\subseteq E^{l}$ for some $l\geq 1$ and $K\cap E^{0}=\emptyset$,  
\item\label{item3E} there exists $W\subseteq E^{1}$ with $s$ injective on $W$ and $x_{1}\in W$ for all $x\in K\cup U$, and 
\item\label{item4E} $\sigma(Z(U) \setminus Z(K)) = Z(\sigma(U)) \setminus Z(\sigma(K))$. 
\end{enumerate}
\end{lemma}

\begin{proof}
Remark first that $\partial E \setminus E^{0}= Z(E^{1})$ is an open set and that \eqref{item3E} implies \eqref{item4E}, because when $s$ is injective on $W$ then $\sigma$ is injective on $Z(W)$, and \eqref{item4E} then follows since $Z(U), Z(K) \subseteq Z(W)$ .

Now consider a compact set $K \subseteq E^{*}$, an open set $U\subseteq E^{*}$ and an element $x\in (\partial E \setminus E^{0} ) \cap (Z(U)\setminus Z(K))$. Since $K\subseteq E^{*}$ is compact, there exists a $n\in \mathbb{N}$ such that all paths in $K$ has length less than or equal to $n$. We can assume without loss of generality that there exists a $l\in \mathbb{N}_{0}$ with $U \subseteq E^{l}$. If $U\subseteq E^{0}$ we can replace it by $r^{-1}(U)\subseteq E^{1}$, so we can also assume that $l\geq 1$. Since $Z(U E^{k}) \subseteq Z(U)$ we can also assume that either $ n \leq l \leq |x|$ or $l=|x|$, depending on whether $|x|\geq n$ or $|x| <n$.

\bigskip

Assume first that $l \geq n$. For each $1\leq k \leq l$ we choose an open neighbourhood $V_{k} \subseteq E^{1}$ of $x_{k}$ such that $\overline{V_{k}}$ is compact and $\overline{V_{k}} \subseteq W_{k}$ where $W_{k} \subseteq E^{1}$ is open and $s:W_{k} \to E^{0}$ is injective. Set
$$
U'=  U \cap (V_{1}V_{2}\cdots V_{l} ) \setminus (r|_{E^{l}})^{-1}(K\cap E^{0}) \; .
$$
Since $K\cap E^{0}$ is closed and $r|_{E^{l}}:E^{l}\to E^{0}$ is continuous then $U'$ is open. Clearly $x_{1} \cdots x_{l}\in U'$. Now set
$$
K'= (K\setminus E^{0}) \cap \big( \bigcup_{i=1}^{l} \overline{V_{1}} \cdots \overline{V_{i}} \big) \; ,
$$
then $U'$ and $K'$ satisfy \eqref{item1E},\eqref{item2E} and \eqref{item3E} with $W=W_{1}$. An element $y\in Z(U')\cap Z(K)$ can be written $y=uy'$ with $u\in U'$, and since $l\geq n$ then $u=\kappa u'$ with $\kappa \in K$. Since $r(u)\notin K\cap E^{0}$ then $|\kappa|>0$, and hence $\kappa \in K'$. This implies that $y\in Z(K')$. Hence
$$
x\in Z(U') \setminus Z(K') \subseteq (\partial E \setminus E^{0} ) \cap (Z(U)\setminus Z(K)) \; ,
$$
which proves the lemma in this case.

\bigskip

Assume then that $l=|x|$ and define $U'$ as before. For each $1\leq k\leq l$ define 
$$
N_{k}=\overline{V_{1}} \; \overline{V_{2}}  \cdots \overline{V_{k}} \cap (K \cap E^{k}) \; .
$$
If $k>l$ let $K_{i}$ be the projection of $K\cap E^{k}$ onto the $i$'th coordinate of $E^{k}$ for $i\leq k$. Then $K_{i}$ is compact, and we set
$$
N_{k} =(\overline{V_{1}}  \;\cdots \overline{V_{l}} \;  K_{l+1} \cdots  K_{k}) \cap (K\cap E^{k}) \; ,
$$
and we define $K'= N_{1} \cup \cdots \cup N_{n}$. It is straightforward to check that $U'$ and $K'$ satisfies  \eqref{item1E}-\eqref{item3E} and that $x\in Z(U') \setminus Z(K')$. To see that
\begin{equation} \label{eqkappa}
Z(U') \setminus Z(K') \subseteq (\partial E \setminus E^{0} ) \cap (Z(U)\setminus Z(K)) \; ,
\end{equation}
take an arbitrary element $y\in Z(U')\cap Z(K)$, and write $y=\kappa y'$ with $\kappa\in K$ and $y= uy''$ with $u\in U'$. If $|\kappa| \leq l$ then $u=\kappa u'$ and hence $\kappa\in V_{1} \cdots  V_{|\kappa|}$, implying that $\kappa \in N_{|\kappa|} \subseteq K'$. If $|\kappa|>l$ then $\kappa=u\kappa '$, so since $\kappa \in K$ we get $\kappa \in N_{|\kappa|} \subseteq K'$. In conclusion $Z(U')\cap Z(K) \subseteq Z(K')$ which proves \eqref{eqkappa}.
\end{proof}

We call a subset $M \subseteq \partial E$ \emph{invariant under $\sigma$}, or just \emph{invariant}, if it satisfies
\begin{equation} \label{eqdefinv}
\sigma^{-1}(M) = M\setminus E^{0} \; .
\end{equation}
We will need the following result on the Borel sets of $\partial E$, which in particular implies that $E^{*}_{sng}$ is an invariant Borel set of $\partial E$.

\begin{lemma} \label{lemma11}
Let $S\subseteq E^{0}$ be a Borel set. Then the set 
$$
E_{sng}^{*} S = \{ \alpha \in E^{*}_{sng} \; | \; s(\alpha) \in S \}
$$ 
is an invariant Borel subset of $\partial E$. 
\end{lemma}

\begin{proof}
The set $E^{n}$ is open in $E^{*}$ for any $n\in \mathbb{N}_{0}$, from which it follows that $ E_{sng}^{n}= Z(E^{n})\setminus Z(E^{n+1})$ is a Borel set in $\partial E$. If $V\subseteq E^{0}$ is open then
$$
\{ \alpha \in E^{n}_{sng} \; | \; s(\alpha) \in V \}
=E^{n}_{sng} \cap Z((s|_{E^{n}})^{-1}(V) ) \; ,
$$
and hence $\{ \alpha \in E^{n}_{sng} \; | \; s(\alpha) \in S \}$ is Borel for all $n\in \mathbb{N}_{0}$, proving that $E_{sng}^{*} S$ is Borel.

To see that the set is invariant, notice that if $\alpha \in E_{sng}^{*}S$ with $|\alpha|\geq 1$, then $s(\sigma(\alpha))=s(\alpha)\in S\cap E_{sng}^{0}$, implying that $\sigma(\alpha) \in E_{sng}^{*} S$. If $|\alpha|\geq 1$ and $\sigma(\alpha)\in E^{*}_{sng}S$ then $s(\alpha)=s(\sigma(\alpha))\in S\cap E_{sng}^{0}$. Hence $\alpha \in  E^{*}_{sng} S$ which proves the equality \eqref{eqdefinv} with $M=E_{sng}^{*} S$.
\end{proof}

\subsection{Graph C$^{*}$-algebras and their groupoid picture}

If $E$ is a topological graph we can define a left and right action of $C_{0}(E^{0})$ on $C_{c}(E^{1} )$ by $a\cdot \xi = (a \circ r) \xi$ and $\xi\cdot a = \xi(a\circ s)$ for $a\in C_{0}(E^{0})$ and $\xi \in C_{c}(E^{1})$. This makes $C_{c}(E^{1} )$ into a bimodule. For $\xi, \eta \in C_{c}(E^{1})$ then
$$
E^{0}\ni v\mapsto \langle \xi, \eta\rangle(v) = \sum_{e\in s^{-1}(v)} \overline{\xi(e)}\eta(e) 
$$
defines a $C_{0}(E^{0})$ valued inner product on $C_{c} (E^{1})$. One obtains a C$^{*}$-correspondence $H(E)$ by completing $C_{c}(E^{1})$ with respect to this inner product. The graph C$^{*}$-algebra $C^{*}(E)$ associated to the graph $E$ is then the Cuntz-Pimsner algebra of the C$^{*}$-correspondence $(H(E), C_{0}(E^{0}) )$. We remind the reader that our assumption of $E^{0}$ and $E^{1}$ being second countable is equivalent with $C^{*}(E)$ being separable, c.f. Proposition 6.3 in \cite{K1}.

\bigskip

For the analysis we will carry out in this paper, it is crucial that any separable topological graph C$^{*}$-algebra $C^{*}(E)$ can be realised as the groupoid C$^{*}$-algebra $C^{*}(\mathcal{G}_{E})$ of a second countable locally compact Hausdorff \'etale groupoid $\mathcal{G}_{E}$. Following \cite{Y} we construct a groupoid $\mathcal{G}_{E}$ which as a set is defined as
$$
\mathcal{G}_{E}:= \left \{  (\alpha, k-l , \beta) \in \partial E \times\mathbb{Z} \times \partial E \; | \;  k,l \geq 0, \ |\alpha|\geq k, \ |\beta| \geq l \text{ and } \sigma^{k}(\alpha)=\sigma^{l}(\beta)   \right\} \; .
$$
The groupoid $\mathcal{G}_{E}$ is a locally compact Hausdorff \'etale groupoid when we equip it with the topology generated by sets
$$
Z(U, m, n , V) = \left \{  (\alpha, m-n, \beta)\in \mathcal{G}_{E} \; | \; \alpha \in U, \ \beta \in V \text{ and } \sigma^{m}(\alpha)=\sigma^{n}(\beta)  \right \}
$$
where $m, n \in \mathbb{N}_{0}$ and $U\subseteq \sigma^{-m}(\partial E)$ and $V\subseteq \sigma^{-n}(\partial E)$ are open sets satisfying that respectively $\sigma^{m}$ and $\sigma^{n}$ are homeomorphisms on them. Endowed with this topology we can identify the topological space $\partial E$ with the unit space $\mathcal{G}_{E}^{(0)}$ via the map $\partial E \ni \alpha \mapsto (\alpha, 0, \alpha) \in \mathcal{G}_{E}^{(0)}$, and the gauge-action $\{\gamma_{z} \}_{z\in \mathbb{T}}$ on $C^{*}(E)$ can be described as the dual action of the homomorphism $\Phi: \mathcal{G}_{E} \to \mathbb{Z}$ given by $\Phi(\alpha, k, \beta)=k$. This groupoid becomes amenable \cite{RW} and $C^{*}(\mathcal{G}_{E}) \simeq C^{*}(E)$.

\subsection{Tracial weights}

We will apply results that were developed for describing KMS states for C$^{*}$-dynamical systems arising from continuous $1$-cocycles on \'etale groupoids. These were first developed by Renault for principal groupoids \cite{Rbook}, and later generalised by Neshveyev to non-principal groupoids \cite{N}. These results in particular applies to tracial states. We will in the following give a recap of these result for tracial weights, see \cite{C} for the generalisation from states to weights.

Recall that a weight $\psi$ on a C$^{*}$-algebra $\mathcal{A}$ is a map $\psi: \mathcal{A}_{+} \to [0, \infty]$ with $\psi(a+b)=\psi(a)+\psi(b)$ and $\psi(\lambda a) = \lambda \psi(a)$ for all $\lambda \geq 0$. We call $\psi$ a tracial weight if it is densely defined, lower semi-continuous and $\psi(a^{*}a)=\psi(aa^{*})$ for all $a\in \mathcal{A}$. A tracial weight $\psi$ is \emph{extremal} when any decomposition $\psi=\psi_{1}+\psi_{2}$ into tracial weights satisfies $\psi_{1}, \psi_{2} \in \{ \lambda \psi \; | \; \lambda \geq 0 \}$.

\bigskip

Let us in the following elucidate the close connection between tracial weights on groupoid C$^{*}$-algebras and so-called invariant measures on the unit space of the groupoid. Let $\mathcal{G}$ be a locally compact second countable Hausdorff \'etale groupoid with unit space $\mathcal{G}^{(0)}$, and denote the range map by $r: \G\to \Gu$ and the source map by $s: \G \to \Gu$. Recall that a bisection $W \subseteq \mathcal{G}$ is an open set such that $r(W)$ and $s(W)$ are open sets and $r|_{W}:W\to r(W)$ and $s|_{W} : W \to s(W)$ are homeomorphisms. We call a measure $\nu$ on $\Gu$ an \emph{invariant measure} if it is a regular\footnote{We refer to \cite{Cohn} for the definition of a regular measure. We will only consider locally compact second countable Hausdorff spaces, so our measures are regular if and only if they are finite on compact sets.} Borel measure on $\mathcal{G}^{(0)}$ which satisfies
\begin{equation} \label{eqinvgroupoidmeasure}
\nu(s(W)) = \nu(r(W))\; ,
\end{equation}
for all open bisection $W \subseteq \mathcal{G}$. For any tracial weight $\psi$ on $C^{*}(\G)$ the restriction of $\psi$ to $C_{0}(\Gu)$ gives rise to an invariant measure $\nu_{\psi}$ on $\mathcal{G}^{(0)}$, and any invariant measure $\nu$ on $\Gu$ gives rise to a tracial weight via the formula
$$
C^{*}(\G)_{+} \ni \; a \mapsto \int_{\Gu} P(a) \; \mathrm{d} \nu
$$
where $P: C^{*}(\G) \to C_{0}(\Gu)$ is the conditional expectation, c.f. Theorem 6.3 in \cite{C}. We call an invariant measure $\nu$ \emph{extremal} if any decomposition $\nu=\nu_{1}+\nu_{2}$ into invariant measures satisfies that $\nu_{1}, \nu_{2} \in \{\lambda \nu \; | \; \lambda \geq 0 \}$. We call it \emph{ergodic} if any invariant Borel set $B$ of $\mathcal{G}^{(0)}$, i.e. a set with $s(r^{-1}(B))=B$, satisfies either $\nu(B)=0$ or $\nu(B^{C})=0$. It then follows from e.g. Theorem 5.5 in \cite{C} that for an invariant measure $\nu$ then
\begin{equation}\label{eqinviserg}
\text{$\nu$ is an ergodic measure if and only if $\nu$ is an extremal measure.}
\end{equation}

\bigskip

Let us now consider the case of the second countable locally compact Hausdorff \'etale groupoid $\mathcal{G}_{E}$ arising from a second countable topological graph $E$. For this groupoid it is easy to see that $B\subseteq \mathcal{G}_{E}^{(0)} \cong \partial E$ is invariant in the sense that $s(r^{-1}(B))=B$ if and only if is invariant under $\sigma$ as in \eqref{eqdefinv}. If we let $\nu_{\psi}$ be the invariant measure on $\Gu$ corresponding to a tracial weight $\psi$, then the map $\psi \mapsto \nu_{\psi}$ is a surjection from the set of extremal tracial weights on $C^{*}(\G_{E})$ to the set of ergodic invariant measures on $\partial E$, c.f. Corollary 7.5 in \cite{C}. Furthermore, a regular Borel measure $\nu$ on $\partial E$ satisfies \eqref{eqinvgroupoidmeasure} for $\mathcal{G}_{E}$ if and only if it satisfies the condition:
\begin{equation} \label{eqinvmeasure}
\nu(\sigma(A)) = \nu(A) \quad \text{ for all Borel $A\subseteq \partial E\setminus E^{0}$ with $\sigma$ injective on $A$,}
\end{equation}
see e.g. (the proof of) Lemma 3.2 in \cite{T1}.

For the groupoid $\mathcal{G}_{E}$ we can use Theorem 7.4 in \cite{C} with the groupoid homomorphism $c=0$ and $\Phi: \mathcal{G}_{E} \to \mathbb{Z}$ to get the following description of the tracial weights on $C^{*}(E)$.

\begin{thm} \label{thmgroupoiddescrip}
Let $E$ be a second countable topological graph. Define for all $\alpha \in \partial E$ the set
$$
\text{Per}(\alpha) := \{k-l \in \mathbb{Z} \; | \;  k, l\leq |\alpha| \text{ and } \sigma^{k}(\alpha)=\sigma^{l}(\alpha) \} \; .
$$
Let $\nu$ be an extremal non-zero invariant measure on $\partial E$. There exists a unique subgroup $H$ of $\mathbb{Z}$ with
$$
\nu(\{ \alpha \in \partial E \; | \; H=\text{Per}(\alpha)  \}^{C}) =0 \; ,
$$ 
and there is an affine bijection between the state space of $C^{*}(H)$ and the set of tracial weights $\psi$ on $C^{*}(\mathcal{G}_{E})$ satisfying that $\nu_{\psi}=\nu$. A state $\phi$ on $C^{*}(H)$ maps to a tracial weight $\psi_{\nu, \phi}$ satisfying
$$
\psi_{\nu, \phi}(f) := \int_{\partial E} \sum_{(\alpha, k-l, \alpha) \in \mathcal{G}_{E}} f((\alpha, k-l, \alpha)) \phi(u_{k-l}) \; \mathrm{d} \nu(\alpha) \quad \text{ for all $f\in C_{c}(\mathcal{G}_{E})$,}
$$
where $\{u_{g}\}_{g\in H}$ are the canonical unitary generators of $C^{*}(H)$.
\end{thm}

\section{Invariant measures and their Riesz decomposition}
\label{subsection22}

Having established the necessary preliminaries, we will now study the invariant measures on the boundary path space $\partial E$.

Assume $\varphi:X\to Y$ is a local homeomorphism between two locally compact Hausdorff spaces $X$ and $Y$. For $f\in C_{c}(X)$ the function $Y\ni y\mapsto \sum_{x\in \varphi^{-1}(y)} f(x)$ is an element of $C_{c}(Y)$, c.f. Lemma 2.5 in \cite{S}. It follows that any regular measure $\mu$ on $Y$ gives rise to a regular measure $\varphi^{*}\mu$ on $X$ via this map.

Let $E=(E^{0}, E^{1}, r,s)$ be a topological graph. Since $s:E^{1} \to E^{0}$ is a local homeomorphism, any regular Borel measure $\mu$ on $E^{0}$ gives rise to a regular Borel measure $s^{*}\mu$ on $E^{1}$ satisfying that
$$
\int_{E^{1}} h \; \mathrm{d} s^{*} \mu = \int_{E^{0}} \sum_{e\in s^{-1}(v)} h(e) \; \mathrm{d} \mu(v) \quad \text{ for all $h\in C_{c}(E^{1})$.}
$$

\begin{defn} \label{eqvertexinv}
Let $\mu$ be a regular Borel measure on $E^{0}$. We call $\mu$ a \emph{vertex-invariant measure} if it satisfies
$$
\int_{E^{1}} f \circ r \; \mathrm{d} s^{*} \mu \leq \int_{E^{0}} f \; \mathrm{d} \mu
$$
for all positive functions $f\in C_{c}(E^{0})$, and with equality when $\text{supp}(f) \subseteq E_{reg}^{0}$.
\end{defn}

Notice that our vertex-invariant measures were introduced as \emph{invariant measure} in Definition 4.1 in \cite{S}, except that we consider regular measures where \cite{S} requires that the measures are probability measures. Since we also use the term invariant measure for the $\sigma$-invariant measures on the boundary path space $\partial E$, c.f. \eqref{eqinvmeasure}, we will use the terminology vertex-invariant measure for a measure satisfying Definition \ref{eqvertexinv}. It was established in Proposition 4.4 in \cite{S} that a vertex-invariant probability measure gives rise to an invariant probability measure on the boundary path space $\partial E$. In the following we will generalise this result to all regular measures.

\begin{prop}[\cite{S}] \label{prop44schaf}
Let $E$ be a second countable topological graph and let $\mu$ be a vertex-invariant measure on $E^{0}$. There exists a regular invariant measure $\nu$ on $\partial E$ such that
$$
\int_{\partial E} f \circ r \; \mathrm{d} \nu = \int_{E^{0}} f \; \mathrm{d} \mu \quad \text{ for all $f\in C_{c}(E^{0})$.}
$$
\end{prop}

We will provide a new proof of this result, in part to handle the more general case of regular measures and in part because our new proof will provide a better understanding of regular invariant measures, which will be important in our later analysis. The first step in this new proof is to decompose the vertex-invariant measures. For the statement of the next Theorem, we remind the reader that the notation $\varphi_{*}\mu$ denotes the push-forward measures of $\mu$ under the map $\varphi$.

\begin{thm} \label{thmuniqueinvdecomp}
Let $E$ be a second countable topological graph and let $T$ denote the map $\mu \mapsto r_{*}(s^{*}(\mu))$ which maps regular Borel measures on $E^{0}$ to Borel measures on $E^{0}$. 

Let $\mu$ be a vertex-invariant measure on $E^{0}$. Then $\mu$ has a unique decomposition $\mu=\mu_{1}+\mu_{2}$ into two vertex-invariant measures $\mu_{1}$ and $\mu_{2}$ where 
$$
T\mu_{1} = \mu_{1} \quad \text{ and } \quad \lim_{n\to \infty} \int_{E^{0}} f \; \mathrm{d} T^{n}\mu_{2} = 0 \quad  \text{ for all $f\in C_{c}(E^{0})$.}
$$
\end{thm}

\begin{proof}
Assume that $f\in C_{c}(E^{0})$ is positive and $\mu$ is vertex-invariant. We then claim that
$$
\int_{E^{0}} f \; \mathrm{d} (T^{n}\mu) 
\leq \int_{E^{0}} f \; \mathrm{d} (T^{n-1}\mu)  
$$
for all $n\in \mathbb{N}$. By assumption it is true for $n=1$. Assume that it holds true for all $k\leq n$ for some $n\in \mathbb{N}$. Then $T^{n} \mu$ and $T^{n-1} \mu$ are regular and $T^{n} \mu \leq T^{n-1} \mu$, so
\begin{align*}
\int_{E^{0}} f \; \mathrm{d} (T^{n+1}\mu)
&= \int_{E^{0}} \sum_{e\in s^{-1}(v)} f\circ r(e) \; \mathrm{d} (T^{n}\mu)(v)
\leq \int_{E^{0}} \sum_{e\in s^{-1}(v)} f\circ r(e) \; \mathrm{d} (T^{n-1}\mu) (v)\\
&= \int_{E^{0}} f \; \mathrm{d} (T^{n}\mu) \; ,
\end{align*}
proving the claim. It follows, that $T^{n} \mu$ is a regular measure for all $n\in \mathbb{N}_{0}$ when $\mu$ is a vertex-invariant measure. As a consequence, we can define a regular measure $\mu_{1}$ with $\mu_{1} \leq \mu$ by setting
$$
\int_{E^{0}} f \; \mathrm{d}\mu_{1} = \lim_{n\to \infty} \int_{E^{0}} f \; \mathrm{d}(T^{n}\mu) \quad \text{ for $f\in C_{c}(E^{0})$.}
$$
Since $\mu_{1} \leq \mu$ and 
$$
\int_{E^{0}} \sum_{e\in s^{-1}(v)} f \circ r(e) \; \mathrm{d} \mu <\infty
$$ 
for $f\in C_{c}(E^{0})$ one can use dominated convergence to prove
$$
\int_{E^{0}} \sum_{e\in s^{-1}(v)} f \circ r(e) \; \mathrm{d}\mu_{1} = \lim_{n\to \infty} \int_{E^{0}} \sum_{e\in s^{-1}(v)} f \circ r(e) \; \mathrm{d}(T^{n}\mu)
$$
for $f\in C_{c}(E^{0})$. This implies that $T\mu_{1} = \mu_{1}$. Set $\mu_{2}=\mu-\mu_{1}$, it is then straightforward to check that $\mu_{1}$ and $\mu_{2}$ are vertex-invariant. For any function $f\in C_{c}(E^{0})$ and $\varepsilon >0$ there then exists a $N\in \mathbb{N}$ such that 
$$
\left | \int_{E^{0}} f  \; \mathrm{d}(T^{n}\mu_{2}) \right |=
\left | \int_{E^{0}} f  \; \mathrm{d}(T^{n}\mu) - \int_{E^{0}} f  \; \mathrm{d}\mu_{1} \right | \leq \varepsilon \quad \text{ for all $n\geq N$.} 
$$
This proves that the decomposition exists. We leave it to the reader to check that the decomposition is unique.
\end{proof}

Using Theorem \ref{thmuniqueinvdecomp} we can now provide a new proof of Proposition \ref{prop44schaf}.

\begin{proof}[proof of Proposition \ref{prop44schaf}]
Assume first that $\mu$ is a  vertex-invariant measure satisfying that 
\begin{equation} \label{eeefoerst}
\lim_{n\to \infty} \int_{E^{0}} f \; \mathrm{d} T^{n}\mu =0 \text{ for all $f\in C_{c}(E^{0})$}. 
\end{equation}
We define a regular positive measure $\eta$ on $E^{0}$ by
$$
\int_{E^{0}} f\; \mathrm{d} \eta :=
\int_{E^{0}} f \; \mathrm{d} \mu  - \int_{E^{1}} f \circ r \; \mathrm{d} s^{*} \mu \quad \text{ for $f\in C_{c}(E^{0})$.}
$$
Since $E_{sng}^{n}$ with the subspace topology from $E^{n}$ is homeomorphic with the Borel subset $E_{sng}^{n}$ of $\partial E$. Let $s_{n}: E^{n} \to E^{0}$ be the source map, since it is a local homeomorphism we can consider $s_{n}^{*} \eta$ as a Borel measure on $\partial E$ supported on $E^{n}_{sng}$. We now claim that
$$
\nu:=\sum_{i=0}^{\infty} s_{n}^{*} \eta
$$
is a regular invariant measure on $\partial E$ such that $r_{*}\nu = \mu $. Let $f\in C_{c}(E^{0})$ be positive. The map $f\circ r$ is not necessarily compactly supported, but by approximating it by compactly supported functions from below and using monotone convergence, we get for any $i\in \mathbb{N}_{0}$ that
$$
\int_{E^{0}} \sum_{\alpha \in s_{i}^{-1}(v) } f\circ r(\alpha) \; \mathrm{d} T\mu(v)=
\int_{E^{0}} \sum_{e\in s^{-1}(v)}\sum_{\alpha \in s_{i}^{-1}(r(e)) } f\circ r(\alpha) \; \mathrm{d} \mu(v)=
\int_{E^{0}} \sum_{\beta \in s_{i+1}^{-1}(v) } f\circ r(\beta) \; \mathrm{d} \mu(v) \; , 
$$
from which it follows that 
\begin{align*}
\int_{\partial E} f\circ r \; \mathrm{d} s_{i}^{*} \eta
&=\int_{E^{0}} \sum_{\alpha \in s_{i}^{-1}(v) } f\circ r(\alpha) \; \mathrm{d} \eta(v) \\
&= \int_{E^{0}} \sum_{\alpha \in s_{i}^{-1}(v) } f\circ r(\alpha) \; \mathrm{d} \mu(v)
-\int_{E^{0}} \sum_{\alpha \in s_{i}^{-1}(v) } f\circ r(\alpha) \; \mathrm{d} T\mu(v) \\
&= \int_{E^{0}} f \; \mathrm{d} T^{i}\mu
- \int_{E^{0}} f \; \mathrm{d} T^{i+1}\mu \; . 
\end{align*}
By assumption this implies that
$$
\int_{\partial E} f\circ r \; \mathrm{d} \nu = \int_{E^{0}} f \; \mathrm{d} \mu -\lim_{n\to \infty} \int_{E^{0}} f \; \mathrm{d} T^{n} \mu = \int_{E^{0}} f\; \mathrm{d} \mu \; .
$$
If $S\subseteq E^{0}$ it is easy to check that $1_{S} \circ r = 1_{Z(S)}$, and since $\partial E$ can be covered by open sets $Z(U)$ with $U\subseteq E^{0}$ open and $\mu(U)<\infty$, this equality implies that $\nu$ is a regular measure. Since $s_{i}^{*} \eta$ is concentrated on $E^{i}_{sng}$, we first see that if $A\subseteq E^{\infty}$ is Borel then $\nu(A)=0=\nu(\sigma(A))$. By second countability, it therefore suffices to check that $\nu(A)=\nu(\sigma(A))$ for a Borel set $A \subseteq E^{i}_{sng}$ with $s_{i}$ and $\sigma$ injective on $A$, yet for such a set we have that
$$
\nu(A) = s_{i}^{*}\eta(A) = \eta(s_{i}(A)) =\eta(s_{i-1}(\sigma(A)))= s_{i-1}^{*}\eta( \sigma(A))
=\nu(\sigma(A)) \; ,
$$
which proves the Proposition for measures $\mu$ satisfying \eqref{eeefoerst}.

\bigskip

Assume now that $\mu$ satisfies that $T\mu = \mu$. Following Section 3 in \cite{S} we define $\partial E_{n} := E_{sng}^{0} \cup \cdots \cup E_{sng}^{n-1}\cup E^{n}$ and set 
$$
Z_{n}(S)=\{ \alpha \in \partial E_{n} \; | \; r(\alpha)\in S \text{ or } \alpha_{1} \cdots \alpha_{k} \in S \text{ for some $1\leq k \leq n$} \} 
$$
where $S \subseteq E^{*}$. Then $\partial E_{n}$ equipped with the topology generated by the basis $Z_{n}(U)\setminus Z_{n}(K)$ with $U\subseteq E^{*}$ open and $K \subseteq E^{*}$ compact satisfies that $\rho_{n}:\partial E_{n} \to \partial E_{n-1}$ defined by 
$$
\rho(\alpha_{1} \cdots \alpha_{k})=
\begin{cases}
\alpha_{1} \cdots \alpha_{n-1} & \text{if } k=n \; , \\
\alpha_{1} \cdots \alpha_{k} & \text{if } k<n \; . 
\end{cases}
$$
is a proper, continuous and surjective map, and that we get a homeomorphism
$$
\partial E = \varprojlim_{n} (\partial E_{n} , \rho_{n}) \; ,
$$
c.f. Proposition 3.3 in \cite{S}. Since $\sigma^{n}:\partial E_{n} \cap E^{n} \to E^{0}$ is a local homeomorphism by Proposition 3.5 in \cite{S}, we can define a Borel measure $(\sigma^{n})^{*}(\mu)$ on $\partial E_{n}$ which is concentrated on $E^{n}$.

\bigskip 

\emph{Claim: $r_{*}(\sigma^{n})^{*} \mu = \mu$ for all $n\in \mathbb{N}$.} Assume that $f\in C_{c}(E^{0})$ is positive and $n\in \mathbb{N}$. Then
\begin{align*}
\int_{E^{0}} f \; \mathrm{d} r_{*}(\sigma^{n+1})^{*} \mu 
&= \int_{\partial E_{n+1}} f\circ r \; \mathrm{d} (\sigma^{n+1})^{*} \mu 
= \int_{E^{0}} \sum_{\alpha\in (\sigma^{n+1})^{-1} (v) } f\circ r(\alpha) \; \mathrm{d}  \mu (v)  \\
&= \int_{E^{0}} \sum_{e \in s^{-1}(v)} \sum_{\alpha\in (\sigma^{n})^{-1} (r(e)) } f\circ r(\alpha e) \; \mathrm{d}  \mu (v)\\
&= \int_{E^{0}}  \sum_{\alpha\in (\sigma^{n})^{-1} (v) } f\circ r(\alpha) \; \mathrm{d}  T\mu (v) \\
&= \int_{E^{0}} f \; \mathrm{d} r_{*}(\sigma^{n})^{*} \mu \; ,
\end{align*}
since $T\mu = \mu$. This proves the claim.

\bigskip

Using the claim we can now see that for $U\subseteq E^{0}$ open with compact closure then 
$$
(\sigma^{n})^{*}(\mu)(Z_{n}(U)) = r_{*}(\sigma^{n})^{*}(\mu)(U)=\mu(U)<\infty \; ,
$$ 
and hence $(\sigma^{n})^{*}(\mu)$ is a regular measure on $\partial E_{n}$. For $f\in C_{c}(\partial E_{n-1})$ we get
\begin{align*}
\int_{\partial E_{n}} f\circ \rho_{n} \;  \mathrm{d} (\sigma^{n})^{*} \mu
&= \int_{E^{0}} \sum_{\alpha \in (\sigma^{n})^{-1}(v)} f\circ \rho_{n}(\alpha) \;  \mathrm{d} \mu(v)
= \int_{E^{0}} \sum_{e\in s^{-1}( v)} \sum_{\beta \in (\sigma^{n-1})^{-1}(r(e))} f(\beta) \;  \mathrm{d} \mu(v)\\
&=\int_{E^{0}}  \sum_{\beta \in (\sigma^{n-1})^{-1}(v)} f(\beta) \;  \mathrm{d} T\mu(v)
= \int_{\partial E_{n-1}} f \; \mathrm{d} (\sigma^{n-1})^{*}\mu \; ,
\end{align*}
which proves that $(\rho_{n})_{*} (\sigma^{n})^{*} \mu = (\sigma^{n-1})^{*} \mu$. By Theorem 1 in \cite{OO} there therefore exists a unique regular measure $\nu$ on $\partial E$ such that $\nu \circ p_{n}^{-1} = (\sigma^{n})^{*} \mu$, where $p_{n}$ is the projection from $\partial E$ to $\partial E_{n}$.

To see that $\nu$ is invariant, let $O \subseteq E^{*}$ be an open set such that $\sigma$ is injective on $Z(O) \subseteq \partial E \setminus E^{0}$. Pick a $k\in \mathbb{N}$ and open subsets $Z(U_{i}) \setminus Z(K_{i}) \subseteq Z(O)$ for $i=1, \dots, k$ as in Lemma \ref{lemmanotE}. For sufficiently big $n$ then $p_{n}^{-1}(Z_{n}(U_{i}) \setminus Z_{n}(K_{i})) = Z(U_{i}) \setminus Z(K_{i})$ for all $1\leq i \leq k$, so we have that
\begin{equation} \label{eeeee1}
\nu\left(\bigcap_{i=1}^{k} Z(U_{i}) \setminus Z(K_{i}) \right) = (\sigma^{n})^{*} \mu \left(\bigcap_{i=1}^{k} Z_{n}(U_{i}) \setminus Z_{n}(K_{i})\right)
\end{equation}
and using that $\sigma$ is injective on $Z(O)$ and \eqref{item3E} in Lemma \ref{lemmanotE}, we get that
\begin{align} \label{eeeee2}
\nonumber \nu\left[\sigma(\bigcap_{i=1}^{k} Z(U_{i}) \setminus Z(K_{i}))\right] &=\nu\left[\bigcap_{i=1}^{k} Z(\sigma(U_{i})) \setminus Z(\sigma(K_{i}))\right]   \\
&=(\sigma^{n-1})^{*} \mu \left(\bigcap_{i=1}^{k} Z_{n-1}(\sigma(U_{i})) \setminus Z_{n-1}(\sigma(K_{i}))\right)\\
\nonumber&= (\sigma^{n-1})^{*} \mu  \left( \sigma\left(\bigcap_{i=1}^{k} Z_{n}(U_{i}) \setminus Z_{n}(K_{i})\right) \right)
\; .
\end{align}
If $\sigma$ is injective on a set $A \subseteq \partial E_{n} \setminus E^{0}$, then by second countability we can write 
$$
A=A_{0}\sqcup \bigsqcup_{k\in \mathbb{N}} A_{k}
$$ 
such that $\sigma^{n}$ is injective on $A_{k}$ for $k\geq 1$ and $A_{0} \cap E^{n}=\emptyset$, and hence
$$
(\sigma^{n})^{*}\mu(A) = \sum_{k\in \mathbb{N} } \mu(\sigma^{n}(A_{k}))
= \sum_{k\in \mathbb{N} } (\sigma^{n-1})^{*}\mu(\sigma(A_{k}))
=(\sigma^{n-1})^{*}\mu(\sigma(A)) \; .
$$
Using this observation and \eqref{eeeee2} we see that
$$
\nu\left[\sigma(\bigcap_{i=1}^{k} Z(U_{i}) \setminus Z(K_{i}))\right]
= 
(\sigma^{n})^{*} \mu  \left( \bigcap_{i=1}^{k} Z_{n}(U_{i}) \setminus Z_{n}(K_{i})\right)
=\nu\left[\bigcap_{i=1}^{k} Z(U_{i}) \setminus Z(K_{i})\right]
\; ,
$$
and since intersections of such sets from Lemma \ref{lemmanotE} are a $\pi$-system we have that $\nu$ is $\sigma$-invariant, c.f. Corollary 1.6.4 in \cite{Cohn}.

By linearity and Theorem \ref{thmuniqueinvdecomp} we have now proven Proposition \ref{prop44schaf}.
\end{proof}

\begin{remark}\label{remarkdecomp}
Notice that the invariant measure associated to a vertex-invariant measure $\mu$ with $T\mu=\mu$ is concentrated on $E^{\infty}$ while an invariant measure associated to a vertex-invariant measure $\mu$ with 
$$
\lim_{n\to \infty} \int_{E^{0}} f \; \mathrm{d} T^{n}\mu=0 \quad \text{ for $f\in C_{c}(E^{0})$}
$$
is concentrated on $E^{*}_{sng}$.
\end{remark}

\begin{remark} \label{remarkintegral}
The motivation for providing a new proof is mainly to deduce remark \ref{remarkdecomp}, but there is also a minor detail in the original proof in \cite{S} that is unclear. In the formula (6) of Proposition 4.4 in \cite{S} a function $f\in C_{c}(\partial E_{n+1})$ is integrated over $\partial E_{n+1}\setminus E^{0}$ for a regular measure on $\partial E_{n+1}\setminus E^{0}$. A priori this is not necessarily finite: If $m$ is a regular measure on $U$, and $U$ is an open set in a space $X$, then $m$ might not be a regular measure on $X$. As an example, consider the topological graph $E$ with $E^{0}=[0,1]$, $E^{1}=(0,1)$ and $r$ and $s$ the inclusion maps. It is then clear that $\partial E_{1}=[0,1]$ with its normal topology, and that $\partial E_{1}\setminus E_{0}=(0,1)$. Consider the measure 
$$
m=\sum_{n\geq 2} \delta_{1/n}
$$  
on $(0,1)$ where $\delta_{1/n}$ is the Dirac-measure at $1/n$. Then $m$ is a regular measure on $(0,1)$ since it is finite on compact sets, yet it is not a regular measure on $[0,1]$. In the the proof of Proposition 4.4 in \cite{S} one can remedy this fact by including a regularity statement in the inductive proof of the formula (7).
\end{remark}

It is proven in Theorem 6.3 in \cite{S} that there exists an affine homeomorphisms between the vertex-invariant probability measures and the gauge-invariant tracial states on $C^{*}(E)$. Since a tracial state $\omega$ is gauge-invariant if and only if $\omega=\omega \circ P$ with $P:C^{*}(E)\to C_{0}(\partial E)$ the canonical conditional expectation, one gets an affine bijection between the vertex-invariant probability measures and the invariant probability measures. For regular invariant measures, we provide a direct proof of this fact for second countable graphs.

\begin{prop} \label{propinvmeasurebij}
Let $E=(E^{0}, E^{1}, r, s)$ be a second countable topological graph. The map $\nu \mapsto r_{*} \nu$ defines an affine bijection between the set of invariant measures on $\partial E$ and the set of vertex-invariant measures on $E^{0}$.
\end{prop}

\begin{proof}
Let us first check that the map is well defined, i.e. that $r_{*} \nu$ satisfies Definition \ref{eqvertexinv}. Let therefore $\nu$ be a regular measure on $\partial E$. If $K\subseteq E^{0}$ is compact then $Z(K)$ is compact in $\partial E$, and hence $\nu(r^{-1}(K))=\nu(Z(K))<\infty$. So $r_{*}\nu$ is a regular measure. Let $f\in C_{c}(E^{0})$ be positive, then we get
\begin{align} \label{eqanyinv}
\int_{E^{1}} f \circ r \; \mathrm{d} s^{*}(r_{*}\nu) 
&= \int_{E^{0}} \sum_{e\in s^{-1}(v) } f \circ r (e)\; \mathrm{d} r_{*}\nu (v)
= \int_{\partial E} \sum_{e\in s^{-1}(r(\alpha)) } f \circ r (e)\; \mathrm{d} \nu (\alpha) \nonumber\\ 
&= \int_{\partial E} \sum_{\beta\in \sigma^{-1}(\alpha) } f \circ r (\beta)\; \mathrm{d} \nu (\alpha) \; .
\end{align}
When $\nu$ is invariant on $\partial E$ we know from \eqref{eqinvmeasure} that $\sigma^{*}\nu(A)=\nu(\sigma(A))=\nu(A)$ whenever $\sigma$ is injective on $A$, and hence it is true for all $A\subseteq \partial E\setminus E^{0}$. Since $f\circ r$ might not be zero on $E^{0}$, we get that
$$
\int_{\partial E} \sum_{\beta\in \sigma^{-1}(\alpha) } f \circ r (\beta)\; \mathrm{d} \nu (\alpha)=
\int_{\partial E\setminus E^{0}}  f \circ r \; \mathrm{d} \sigma^{*}\nu 
\leq \int_{\partial E }  f \circ r \; \mathrm{d} \nu 
= \int_{E^{0} }  f  \; \mathrm{d} (r_{*}\nu) \; ,
$$
with equality if $f\circ r$ is zero on $E^{0}_{sng}$. If $\text{supp}(f) \subseteq E_{reg}^{0}$ then $f\circ r$ is zero on $E^{0}_{sng}$, which proves that the map $\nu \mapsto r_{*} \nu$ is well defined. That it is surjective follows from Proposition \ref{prop44schaf}. To check injectivity, assume that $r_{*}\nu_{1} = r_{*}\nu_{2}$ for two invariant measures $\nu_{1}$ and $\nu_{2}$. If $U \subseteq E^{0}$ is Borel we have by definition of $Z(U)$ that
$$
\nu_{1}(Z(U)) = \nu_{1} (r^{-1}(U)) 
= \nu_{2} (r^{-1}(U))
= \nu_{2} (Z(U)) \; .
$$
For a given $n\in \mathbb{N}$ and a Borel set $V \subseteq E^{n}$ satisfying that $\sigma^{n} : V \to \sigma^{n}(V) \subseteq E^{0}$ is injective, we have by invariance of the two measures that
$$
\nu_{1}(Z(V)) = \nu_{1}( \sigma^{n}(Z(V)))
=\nu_{1}( Z(\sigma^{n}(V)))
=\nu_{2}( Z(\sigma^{n}(V))) = \nu_{2}(Z(V)) \; .
$$
If $V_{n} \subseteq E^{n}$ and $V_{m} \subseteq E^{m}$ are two such sets with $m\geq n$ then the set 
$$
V'=\{\alpha \in E^{m} \; | \; \alpha =v\alpha ' \text{ with } v\in V_{n} \}
$$
is open, and since
$$
Z(V_{n})\cap Z(V_{m}) = Z(V_{m}\cap V' ) \; ,
$$
such sets are a $\pi$-system that generates the Borel $\sigma$-algebra. It follows that $\nu_{1} = \nu_{2}$, which finishes the proof.
\end{proof}

\begin{example} \label{ex25}
Assume that $E$ is a second countable discrete topological graph, i.e. that it is a countable directed graph. A measure $\mu$ on $E^{0}$ is then vertex-invariant if and only if it satisfies Definition \ref{eqvertexinv} for all functions on the form $f=1_{\{v\}}$ with $v\in E^{0}$, which is true if and only if
\begin{align*}
\mu(\{v\}) &\geq \int_{E^{1}} 1_{v} \circ r \; \mathrm{d} s^{*}\mu
= \int_{E^{0}} \sum_{e\in s^{-1}(w)} 1_{v} \circ r(e) \; \mathrm{d} \mu(w)
= \sum_{w\in E^{0}} \mu(\{ w \}) \sum_{e\in s^{-1}(w)} 1_{v} \circ r(e) \\
&= \sum_{w\in E^{0}} \mu(\{ w \}) |v E^{1} w|  \; ,
\end{align*}
with equality when $v$ is a regular vertex, i.e. when $0<|r^{-1}(v)| < \infty$. If we interpret $\mu$ as a vector $\mu: E^{0} \to [0, \infty[$ and $A$ is the vertex matrix $A(v,w)=|vE^{1}w|$ for $v,w\in E^{0}$, we get that $\mu$ is vertex-invariant if and only if $(A\mu)_{v} \leq \mu_{v}$ with equality when $v$ is regular. It is a well established fact that invariant measures for directed graphs are in a bijective correspondence with exactly that class of vectors, see e.g. Theorem 2.7 in  \cite{T2} and Section 3 in \cite{Tomforde}. So in the discrete case Proposition \ref{propinvmeasurebij} is well known.

For discrete graphs it is customary to use the Riesz decomposition of the these vectors to analyse them, c.f. e.g.  \cite{T2}. This approach provides the motivation behind Theorem \ref{thmuniqueinvdecomp} and the new proof of Proposition \ref{prop44schaf}.
\end{example}

Using our new proof of Proposition \ref{prop44schaf} we get a decomposition of invariant measures corresponding to the decomposition in Theorem \ref{thmuniqueinvdecomp}.

\begin{prop} \label{propuniqueinvdecomp}
Let $E$ be a second countable topological graph. Let $\nu$ be an invariant measure on $\partial E$, and let $\mu = r_{*}\nu$ be the associated vertex-invariant measure. Then
\begin{enumerate}
\item \label{enuconc1} $\nu$ is concentrated on $E^{\infty}$ if and only if $T\mu = \mu$, and
\item \label{enuconc2} $\nu$ is concentrated on $E^{*}_{sng}$ if and only if $\lim_{n} \int_{E^{0}} f \; \mathrm{d} T^{n}\mu =0$ for all $f\in C_{c}(E^{0})$.
\end{enumerate}
\end{prop}

\begin{proof}
Remark \ref{remarkdecomp} implies the if-implication in both statements. For the only if implication in \eqref{enuconc1}, assume that $\nu$ is concentrated on $E^{\infty}$ and decompose $r_{*}\nu=\mu_{1}+\mu_{2}$ as in Theorem \ref{thmuniqueinvdecomp}. Let $\nu_{i}$ be the invariant measure with $r_{*}\nu_{i}=\mu_{i}$ for $i=1,2$. It follows that $\nu=\nu_{1}+\nu_{2}$ and that $\nu_{2}$ is concentrated on $E^{*}_{sng}$, and hence $\nu_{2}=0$ since $\nu(E^{*}_{sng})=0$. This implies that $\mu_{2}=0$ and proves the only if implication. The only if implication in \eqref{enuconc2} follows similarly.
\end{proof}

It follows from Theorem \ref{thmuniqueinvdecomp} that we have a very natural decomposition of vertex-invariant measures, which corresponds to a very natural decomposition of invariant measures by Proposition \ref{propuniqueinvdecomp}. The following terminology is inspired by the properties of paths in $\partial E$ which the invariant measures are concentrated on, and is in line with the terminology for directed graphs \cite{T2}.

\begin{defn}
We call a vertex-invariant measure $\mu$ on $E^{0}$ a \emph{boundary measure} if 
$$
\lim_{n\to \infty} \int_{E^{0}} f \ \mathrm{d} T^{n}\mu=0 \quad \text{ for all $f\in C_{c}(E^{0})$.}
$$
In accordance with this, we call an invariant measure $\nu$ on $\partial E$ a \emph{boundary measure} if it is concentrated on $E_{sng}^{*}$.

We call a vertex-invariant measure $\mu$ on $E^{0}$ \emph{harmonic} if $T\mu=\mu$. In accordance with this, we call an invariant measure $\nu$ on $\partial E$ \emph{harmonic} if it is concentrated on $E^{\infty}$.
\end{defn}

It follows from \eqref{eqinviserg} that any extremal invariant measures is either a boundary measure or a harmonic measure. We will therefore in the following investigate the boundary measures and the harmonic measures separately.

\section{Boundary measures}\label{subsection22}
We will in the following give a complete description of the extremal boundary measures.

\begin{prop} \label{propsurjbound}
Let $E$ be a second countable topological graph. If $\nu$ is an extremal non-zero  boundary measure on $\partial E$ there exists a vertex $v\in E^{0}_{sng}$ such that for any $w\in E^{0}$ there is an open neighbourhood $V_{w}$ of $w$ in $E^{0}$ with $|V_{w} E^{*} v|<\infty$ and such that $\nu$ is concentrated on the countable set $E^{*} v$.

In particular, this implies that there exists a $r > 0$ such that $\nu$ is on the form
$$
\nu = r \cdot \sum_{\alpha \in E^{*}v} \delta_{\alpha} \; .
$$
\end{prop}

\begin{proof}
Let for each $i \in \mathbb{N}$ the set $U_{i} \subseteq E^{0}$ satisfy that $\overline{U_{i}}$ is compact, $U_{i}$ is open and $E^{0} = \bigcup_{i\in \mathbb{N}} U_{i}$. Define $V_{1}=U_{1}$ and for each $n>1$ define
$$
V_{n} = U_{n} \setminus ( U_{1} \cup \cdots \cup U_{n-1} ) \; .
$$
The sets $\{V_{n}\}_{n\in \mathbb{N}}$ are disjoint with union $E^{0}$ so we get
$$
E^{*}_{sng} = \bigsqcup_{n \in \mathbb{N}} E^{*}_{sng}V_{n} \; .
$$
Since $\nu$ is extremal and hence ergodic, it follows from Lemma \ref{lemma11} that there exists a unique $m \in \mathbb{N}$ such that $\nu((E^{*}_{sng} V_{m} )^{C}) =0$. The relative topology on $\overline{U_{m}}$ makes it a compact second countable Hausdorff space, and hence a metrizable space. Using standard arguments, we can therefore find a sequence $\{W_{i}\}_{i=1}^{\infty}$ of Borel subsets of $\overline{U_{m}}$ with $W_{i+1} \subseteq W_{i}$ such that the diameter of $W_{i}$ is at most $1/i$, and such that $\nu ( (E^{*}_{sng} W_{i} )^{C}) =0$ for all $i$. We now have that
$$
0=  \nu \big ( \bigcup_{i\in \mathbb{N} }(E^{*}_{sng} W_{i})^{C} \big ) 
=\nu \Big ( \big(\bigcap_{i\in \mathbb{N} } E^{*}_{sng} W_{i} \big )^{C} \Big ) 
= \nu \Big (  \Big(E^{*}_{sng} \big(\bigcap_{i\in \mathbb{N} } W_{i} \big) \Big)^{C}\Big )
\; .
$$
Since $\nu \neq 0$ it follows that $\bigcap_{i\in \mathbb{N} } W_{i}= \{v \}$ for some $v\in E^{0}$. Hence we have that $0=\nu((E^{*}_{sng} v)^{C})$, so $v\in E^{0}_{sng}$. Since the topology is second countable and $s$ is a local homeomorphism, the set $E^{*}_{sng} v=E^{*}v$ contains at most countably many elements. Now take a path $\alpha \in E^{*} v$, we then have by \eqref{eqinvmeasure} that $\nu(\{v\})=\nu(\sigma^{|\alpha|}(\{\alpha\}) )=\nu(\{\alpha\})$. This implies that
$$
\nu= \sum_{\alpha \in E^{*}v} \nu(\{\alpha\}) \delta_{\alpha}
= \nu(\{v\}) \sum_{\alpha \in E^{*}v}  \delta_{\alpha} \; ,
$$
where $\nu(\{v\})>0$ by assumption. Since $\nu$ is regular and $Z(K)$ is a compact set for any compact $K\subseteq E^{0}$, we can for any $w\in E^{0}$ find an open neighbourhood $V_{w}$ in $E^{0}$ with $\nu(Z(V_{w})) <\infty$. Since
$$
\nu(Z(V_{w})) = \nu(\{v\}) \sum_{\alpha \in E^{*}v}  \delta_{\alpha} (Z(V_{w})) = \nu(\{v\}) \cdot  | V_{w}  E^{*}v | \; ,
$$
this finishes the proof.
\end{proof} 

Recall for the following theorem, that a ray of measures is a set on the form $\{ r\nu \; | \; r> 0\}$ where $\nu$ is measure.

\begin{thm} \label{thmsingbi}
Let $E$ be a second countable topological graph. There is a bijection between the set of rays of extremal invariant boundary measures on $\partial E$ and the set of vertices $v\in E^{0}_{sng}$ satisfying the following condition: For any $w\in E^{0}$ there exists an open neighbourhood $V_{w} \subseteq E^{0} $ of $w$ satisfying $|V_{w}E^{*}v|< \infty$.

The ray of extremal invariant boundary measures corresponding to such a $v\in E_{sng}^{0}$ is on the form $\{ r \cdot \nu \; |\; r>0\}$ where
$$
\nu = \sum_{\alpha \in E^{*}v} \delta_{\alpha} \; .
$$
\end{thm}

\begin{proof}
Assume that $v\in E^{0}_{sng}$ satisfies that for any $w\in E^{0}$ there exists an open neighbourhood $V_{w} \subseteq E^{0}$ of $w$ satisfying $|V_{w}E^{*}v|< \infty$, and define $\nu:=\sum_{\alpha \in E^{*}v} \delta_{\alpha}$. If $V\subseteq E^{0}$ is an open set then
$$
\nu(Z(V)) = \sum_{\alpha \in E^{*}v} \delta_{\alpha}(Z(V)) =|VE^{*}v| \; .
$$
Since the sets $Z(V_{w})$ with $w\in E^{0}$ is an open cover of $\partial E$ it follows that $\nu$ is a regular measure. Now assume that $A \subseteq \partial E \setminus E^{0}$ with $\sigma$ injective on $A$. Since $\sigma(\alpha) \in E^{*}v$ if and only if $\alpha \in E^{*}v$ for all $\alpha \in A$ we get that
\begin{align*}
\nu(\sigma(A)) &= |\sigma(A) \cap E^{*}v|
=|\sigma (A \cap E^{*} v)|
=|A \cap E^{*}v|
= \nu(A) \; ,
\end{align*}
which by \eqref{eqinvmeasure} proves that $\nu$ is an invariant boundary measure. The measures is clearly extremal, because any invariant set $A$ either contains $E^{*}v$ or is disjoint from it, depending on whether it contains $v$. It follows that the map which maps such a vertex $v$ to the extremal ray of boundary measures $\{ r \cdot \nu \; | \; r>0\}$ is well defined, and by Proposition \ref{propsurjbound} it is also  surjective. To see that it is injective, notice that $\nu(\{w\})>0$ for a $w\in E^{0}$ if and only if $w=v$ which proves that the map is a bijection.
\end{proof}

\begin{remark}
If $E^{0}$ is compact we can exchange the condition: For any $w\in E^{0}$ there exists an open neighbourhood $V_{w} \subseteq E^{0}$ of $w$ satisfying $|V_{w}E^{*}v|< \infty$, with the condition that $|E^{*}v| < \infty$.

It follows from Theorem \ref{thmsingbi} that the rays of extremal finite invariant boundary measures maps bijectively onto the vertexes satisfying the extra condition $|E^{*}v|<\infty$. This gives the Corollary below.
\end{remark}

\begin{cor}
Let $E$ be a second countable topological graph. There is a bijection between the set of extremal invariant boundary probability measures on $\partial E$ and the set of vertices $v\in E^{0}_{sng}$ with $|E^{*}v|< \infty$. The extremal invariant boundary probability measures corresponding to such a vertex is on the form
$$
\nu = \frac{1}{|E^{*}v|}\sum_{\alpha \in E^{*}v} \delta_{\alpha} \; .
$$
\end{cor}

The upshot of Theorem \ref{thmsingbi} is that the extremal boundary measures are easy to describe. To support this statement, let us argue that for simple C$^{*}$-algebras of topological graphs, boundary measures only occurs in the very special cases where the $C^{*}$-algebra of the graph is (isomorphic to) the compact operators $\mathcal{K}(H)$ on a Hilbert space $H$. 

\begin{cor} \label{cor211}
Assume that $E$ is a second countable topological graph and that $C^{*}(E)$ is simple. If there exists an invariant boundary measure $\nu$ on $\partial E$ then there exists a $v\in E^{0}_{sng}$ with $E^{*}v$ countable such that
$$
C^{*}(E) \simeq \mathcal{K}(l^{2}(E^{*}v)) \; .
$$
\end{cor}

\begin{proof}
Set $\eta = \int P(\cdot) \mathrm{d} \nu$ which is a tracial weight on $C^{*}(E)$. Take a $f\in C_{c}(E^{0})_{+}$ with 
$$
\int_{E^{0}} f\; \mathrm{d} T(r_{*}\nu) < \int_{E^{0}}f\; \mathrm{d} (r_{*}\nu) \; ,
$$
which we can rewrite using \eqref{eqanyinv} to get
$$
\int_{\partial E} \sum_{\beta \in \sigma^{-1}(\alpha)} f\circ r(\beta) \; \mathrm{d} \nu(\alpha)  < \int_{\partial E} f\circ r \; \mathrm{d} \nu \; .
$$
Since $r^{-1}(\text{supp}(f))=Z(\text{supp}(f))$ then $f\circ r \in C_{c}(\partial E)_{+}$. We can find an increasing sequence $\{ g_{n} \}_{n\in \mathbb{N}} \subseteq C_{c}(\partial E \setminus E^{0} )$ of positive functions with $\lim_{n\to \infty} g_{n} = f\circ r|_{\partial E \setminus E^{0}}$. Using that $\mathcal{G}_{E}$ is a minimal groupoid, Proposition 6.10 in \cite{C} implies that the tracial weights $\psi$ satisfying $\psi(f\circ r)=\eta (f\circ r)$ is a simplex. Since
$$
\lim_{n\to \infty} \int_{\partial E} \sum_{\beta \in \sigma^{-1}(\alpha)} g_{n}(\beta) \; \mathrm{d} \nu(\alpha) 
= \int_{\partial E} \sum_{\beta \in \sigma^{-1}(\alpha)} f\circ r(\beta) \; \mathrm{d} \nu(\alpha)  < \int_{\partial E} f\circ r \; \mathrm{d} \nu \; ,
$$
and $\alpha \mapsto  \sum_{\beta \in \sigma^{-1}(\alpha)} g_{n}(\beta)$ lies in $C_{c}(\partial E)$, then Choquet theory implies that there is an extremal tracial weight $\psi$ with $\psi(f\circ r)=\eta (f\circ r)$ and 
$$
\int_{\partial E} \sum_{\beta \in \sigma^{-1}(\alpha)} f\circ r(\beta) \; \mathrm{d} \nu_{\psi}(\alpha) 
=\lim_{n\to \infty} \int_{\partial E} \sum_{\beta \in \sigma^{-1}(\alpha)} g_{n}(\beta) \; \mathrm{d} \nu_{\psi}(\alpha) 
 < \int_{\partial E} f\circ r \; \mathrm{d} \nu_{\psi} \; ,
$$
where $\nu_{\psi}$ is the measures associated to $\psi$. This implies that 
\begin{equation*}
\int_{E^{0}} f\; \mathrm{d} T(r_{*}\nu_{\psi}) 
=\int_{\partial E} \sum_{\beta \in \sigma^{-1}(\alpha)} f\circ r(\beta) \; \mathrm{d} \nu_{\psi}(\alpha)
< \psi(f\circ r)=\int_{E^{0}}f\; \mathrm{d} (r_{*}\nu_{\psi})\; ,
\end{equation*}
from which it follows that $\nu_{\psi}$ is an extremal boundary measure.

So we can assume without loss of generality that $\nu$ is an extremal invariant boundary measure. It follows that there exists a $v \in E_{sng}^{0}$ such that $\nu$ is concentrated on the countable set $E^{*}v$. If we can prove that $\partial E = E^{*}v$ this will prove the Corollary, because then $Z(V_{w})=V_{w}E^{*}v$ and hence $\partial E$ is discrete, and the groupoid corresponding to the graph will be isomorphic to the equivalence relation $E^{*}v \times E^{*}v$, which exactly gives rise to the C$^{*}$-algebra $\mathcal{K}(l^{2}(E^{*}v))$. Since $E^{*}v$ is invariant and $C^{*}(E)$ is simple, it suffices to prove that $E^{*}v$ is a closed set to see that $\partial E = E^{*} v$, because a non-trivial closed invariant set in the unit space of a groupoid gives rise to a non-trivial ideal.

Assume now that $\alpha \in \partial E$ with $\alpha \notin E^{*}v$. By existence of $\nu$ there is an open set $U \subseteq E^{0}$ such that $r(\alpha) \in U$ yet $|U E^{*} v| < \infty$. Write $U E^{*} v = \{ \beta^{1}, \dots, \beta^{n} \}$. We consider two cases. First, assume that $r(\alpha_{j})\neq v$ for all $j \in \mathbb{N}$. By definition each $\beta^{i}$ has $s(\beta^{i}) =  v$, so it follows that $\alpha \notin Z(\{\beta^{1}, \dots \beta^{n} \})$. Since $\{\beta^{1}, \dots \beta^{n} \}$ is compact the set $Z(U)\setminus Z(\{\beta^{1}, \dots \beta^{n} \})$ is an open set containing $\alpha$, yet a $\gamma \in E^{*}v \cap Z(U)$ has to be an element of $U E^{*} v$, proving that $Z(U)\setminus Z(\{\beta^{1}, \dots \beta^{n} \})$ is disjoint from $E^{*} v$.

Assume for the other case that $r(\alpha_{i})=v$ for some $i\in \mathbb{N}$. If this was true for more than one $i$ we would have $|U E^{*} v|=\infty$. So this $i$ is unique. It follows that there exists a unique $\beta^{k}\in U E^{*} v$ with $\alpha=\beta^{k}\alpha'$. Since $\alpha \notin E^{*}v$ we have $|\alpha|>|\beta^{k}|$. Now let $W\subseteq E^{|\beta^{k}|+1}$ be an open set containing the path $\alpha_{1}\dots\alpha_{|\beta^{k}|+1}$ such that $r(W)\subseteq U$, which can be arranged since $r:E^{*} \to E^{0}$ is continuous. The set
$$
Z(W)\setminus Z(\{ \beta^{1}, \dots , \beta^{k-1}, \beta^{k+1}, \dots, \beta^{n} \})
$$
is an open set in $\partial E$ containing $\alpha$. Assume for a contradiction that $\gamma$ lies in this open set and $E^{*}v$. By choice of $W$ then $r(\gamma)\in r(W)\subseteq U$, proving that $\gamma=\beta^{i}$ for some $1\leq i \leq n$. Since $\gamma\notin Z(\{ \beta^{1}, \dots , \beta^{k-1}, \beta^{k+1}, \dots, \beta^{n} \})$ then $\gamma=\beta^{k}$. But $\gamma\notin Z(W)$ because $W\subseteq E^{|\beta^{k}|+1}$, which proves that the intersection is empty. In conclusion $E^{*}v$ is a closed set.
\end{proof}

\begin{remark}
If there exists an invariant boundary probability measure in Corollary \ref{cor211}, then $E^{*}v$ is finite and $C^{*}(E)\simeq M_{|E^{*}v|}(\mathbb{C} )$. 
\end{remark}

While the extremal boundary measures are easy to describe in general, describing all harmonic measures is much more complicated, as illustrated in the following example.   

\begin{example} \label{excrossed}
Let $X$ be a second countable locally compact Hausdorff space and let $\varphi$ be a homeomorphism on $X$. Define a topological graph $E_{X}$ by setting $E_{X}^{0}=E_{X}^{1}=X$, $s=id_{X}$ and $r=\varphi$. It then follows from Example 2 in \cite{K1} that $C^{*}(E_{X}) \simeq C(X)\rtimes_{\varphi} \mathbb{Z}$. Since $s^{*}\mu = \mu$ for any regular measure on $X$, it follows from Definition \ref{eqvertexinv} that a regular measure $\mu$ on $X$ is vertex-invariant and harmonic if and only if it satisfies that $\int_{X} f\circ \varphi \; \mathrm{d} \mu=\int_{X} f \; \mathrm{d} \mu$ for all $f\in C_{c}(X)$, i.e. that $\mu$ is a $\varphi$-invariant measure on $X$. It follows that describing the extremal harmonic vertex-invariant measures in this case is equivalent with describing the ergodic $\varphi$-invariant measures on $X$. 
\end{example}

\section{Loops in topological graphs}\label{subsection23a}
As illustrated in Example \ref{excrossed} it is not feasible to obtain formulas for all extremal harmonic measures in a similar manner as was done for boundary measures. It however turns out that we can describe the extremal harmonic measures which gives rise to non gauge-invariant tracial weights much more concretely by analysing loops in the graph. Before turning to this analysis of harmonic measures, we will need some preliminary results on loops in topological graphs.

We call a finite path $\alpha \in E^{*}$ a loop if it satisfies that $|\alpha|\geq 1$ and $s(\alpha)=r(\alpha)$. Any loop $\alpha \in E^{*}$ gives rise to a well defined infinite path $\alpha^{\infty} \in E^{\infty}$ given by
$$
\alpha^{\infty} = \alpha \alpha \alpha \alpha \cdots \; .
$$
For a loop $\alpha$ in $E$ we introduce the notation
$$
E_{\alpha}^{*} = \{ \beta \in E^{*}r(\alpha)\; | \; \text{$\beta$ is not on the form $\beta=\beta'\alpha$} \}  \; .
$$
Similar to our conventions for $E^{*}$ we write $AE_{\alpha}^{*}B$ for the paths $\beta \in E_{\alpha}^{*}$ with $r(\beta)\in A$ and $s(\beta) \in B$ where $A, B \subseteq E^{0}$. We call a loop $\alpha$ \emph{simple} if $r(\alpha_{i})\neq r(\alpha_{j})$ for all $i\neq j$, c.f. Definition 6.5 in \cite{K3}. We will furthermore work with the following two kinds of loops.

\begin{defn} \label{def217a}
Let $E$ be a topological graph. We call a loop $\alpha \in E^{*}$ an \emph{isolated loop} if it satisfies the following two conditions:
\begin{enumerate}
\item[(a)] $\alpha\neq \gamma^{n}$ for any loop $\gamma \in E^{*}$ and any $n\geq 2$, and
\item[(b)] there exists an open neighbourhood $V_{r(\alpha)} \subseteq E^{0}$ of $r(\alpha)$ with $| V_{r(\alpha)} E_{\alpha}^{*} | < \infty$
\end{enumerate}
We call a loop $\alpha$ a \emph{summable loop} if it satisfies (a) and (b')
\begin{enumerate}
\item[(b')] For any $w\in E^{0}$ there exists an open neighbourhood $V_{w} \subseteq E^{0}$ of $w$ with $| V_{w} E_{\alpha}^{*} | < \infty$.
\end{enumerate}
\end{defn}

Clearly a summable loop is isolated. We will prove that an isolated loop is a simple loop.

\begin{lemma} \label{lemmaunfree}
Let $E$ be a topological graph. If $\alpha$ is an isolated loop then there can not exists a loop $\gamma \in r(\alpha)E^{*}$ such that $\gamma \notin \{ \alpha^{n} \}_{n\in \mathbb{N}}$. In particular, an isolated loop is simple.
\end{lemma}

\begin{proof}
Assume for a contradiction that $\gamma \in r(\alpha)E^{*}$ is a loop such that $\gamma \notin \{ \alpha^{n} \}_{n\in \mathbb{N}}$. Write $\gamma = \gamma'\alpha^{n}$ with $n\geq 0$ as big as possible. It follows that $|\gamma'| >0$ and that $\gamma '$ is a loop. We can also decompose $\gamma'= \gamma '' \delta$, where $\delta \in r(\alpha) E^{*}$ is a loop and $r(\delta_{i})\neq r(\alpha)$ for $2\leq i\leq |\delta|$.  Since $n$ was chosen as big as possible we have that $\delta \in E_{\alpha}^{*} $.

We now claim that $\delta^{m} \in E_{\alpha}^{*}$ for all $m\geq 2$. If $|\delta| \geq |\alpha |$ this follows directly from the fact that $\delta \in E_{\alpha}^{*} $. If $|\delta|<|\alpha|$ and $\delta^{m}=\beta \alpha$, then we can assume (by possibly picking a smaller $m$) that $|\beta | < |\delta|$. By assumption $(a)$ then $|\beta|\neq 0$, which means that $\delta=\beta \gamma$ with $\gamma$ and $\beta$ loops, which can not be the case since $r(\delta_{i})\neq r(\alpha)$ for $2\leq i\leq |\delta|$. So $\{ \delta^{m}\}_{m\in \mathbb{N}} \subseteq V_{r(\alpha)}E_{\alpha}^{*}$, which contradicts assumption $(b)$.

If $\alpha$ is not simple then there exists a $i<j$ such that $r(\alpha_{i})=r(\alpha_{j})$, but then $\gamma:=\alpha_{1}\cdots \alpha_{i-1}\alpha_{j}\cdots \alpha_{|\alpha|} \in r(\alpha)E^{*}$ is a loop such that $\gamma \notin \{ \alpha^{n}\}_{n\in \mathbb{N}}$.
\end{proof}

\begin{defn} [\cite{K3}] 
A free topological graph is a graph $E$ such that there does not exist an element $v\in E^{0}$ satisfying all of the following three conditions:
\begin{enumerate}
\item[(F1)]\label{F1} There exists a simple loop $l_{1}\cdots l_{n}$ in $E$ with $r(l_{1})=v$.
\item[(F2)]\label{F2} If $e\in E^{1}$ with $s(e)\in \{r(\beta) \; | \; \beta \in E^{*}v\}$ and $r(e)=r(l_{k})$ for some $k$ then $e=l_{k}$.
\item[(F3)]\label{F3} $v$ is isolated in $\{r(\beta) \; | \; \beta \in E^{*}v\}$.
\end{enumerate}
\end{defn}
This notion of freeness introduced by Katsura is a simultaneous generalisation of the notion of condition $(K)$ for directed graphs and the notion of freeness for crossed products by homeomorphisms.

\begin{lemma} \label{lemma44}
Let $E$ be a second countable topological graph. Then $E$ is free if and only if there does not exist an isolated loop.
\end{lemma}

\begin{proof}
Assume that $E$ is not a free graph and let $\alpha \in vE^{*}$ be the simple loop from $(F1)$. Since $\alpha$ is simple it clearly satisfies $(a)$ in Definition \ref{def217a}. By $(F3)$ we can find an open set $V_{v} \subseteq E^{0}$ such that $r(\beta)\in V_{v}$ and $\beta \in E^{*}v$ implies that $r(\beta)=v$. Now assume for a contradiction that $\gamma \in V_{v} E_{\alpha}^{*}$ with $|\gamma| \neq 0$. Then $r(\gamma)\in V_{v}$ and $\gamma \in E^{*}v$, and hence $\gamma$ is a loop in $vE^{*}$. Then $s(\gamma_{1})\in \{r(\beta) \; | \; \beta \in E^{*}v\}$ and $r(\gamma_{1})=r(\alpha_{1})$, so by $(F2)$ then $\gamma_{1}=\alpha_{1}$. If $|\gamma|\geq 2$ then $s(\gamma_{2})\in \{r(\beta) \; | \; \beta \in E^{*}v\}$ and $r(\gamma_{2})=s(\gamma_{1})=s(\alpha_{1})=r(\alpha_{2})$, so $\gamma_{2}=\alpha_{2}$. We can continue like this to write $\gamma=\alpha^{n}\alpha_{1}\cdots \alpha_{k}$ for some $0\leq k <|\alpha|$. But $s(\gamma)=r(\alpha)$, so since $\alpha$ is simple, we get that $k=0$. Since $|\gamma|\neq 0$ and $\gamma \in E_{\alpha}^{*}$ we reach a contradiction. So $V_{v} E_{\alpha}^{*} = \{v\}$, and hence $\alpha$ also satisfies $(b)$ in Definition \ref{def217a} which makes it an isolated loop.

Assume now that there exists an isolated loop $\alpha$ in $E^{*}$. Set $v=r(\alpha)$. By Lemma \ref{lemmaunfree} $\alpha$ is simple. Let now $e\in E^{1}$ satisfy that $s(e)=r(\beta)$ with $\beta \in E^{*}v$ and $r(e)=r(\alpha_{k})$ for some $k$. It follows that we have a loop $\alpha_{1}\cdots \alpha_{k-1}e\beta \in r(\alpha)E^{*}$. By Lemma \ref{lemmaunfree} then $\alpha_{1}\cdots \alpha_{k-1}e\beta=\alpha^{n}$. So $e=\alpha_{k}$, which proves $(F2)$. Lastly, let us prove that $v$ is isolated in $\{r(\beta) \; | \; \beta \in E^{*}v\}$. By assumption there exists an open neighbourhood $V_{v}$ of $v$ in $E^{0}$ with $|V_{v} E^{*}_{\alpha} |< \infty$. It follows that 
$$
N=\{r(\beta) \; | \; \beta \in V_{v} E^{*} v \text{ with } r(\beta)\neq v \}
$$ 
is a finite set in $E^{0}$, and hence the set $V_{v} \setminus N$ is an open set in $E^{0}$ containing $v$, yet which does not contain any other elements of $\{r(\beta) \; | \; \beta \in E^{*}v\}$. In conclusion the vertex $v$ ensures that $E$ is not a free graph which proves the Lemma.
\end{proof}

We say that two loops $\alpha$ and $\beta$ are equivalent if there exists a $k$ such that
$$
\beta = \alpha_{k+1}\alpha_{k+2} \cdots \alpha_{|\alpha|}\alpha_{1} \cdots \alpha_{k} \; .
$$
This condition is equivalent to having $|\alpha|=|\beta|$ and $\sigma^{k}(\alpha^{\infty} )=\beta^{\infty}$ for some $k$. The following observation will become vital later.

\begin{lemma} \label{lemmaequiv}
Let $\alpha$ and $\beta$ be two equivalent loops. If $|VE_{\beta}^{*}|<\infty$ then $|VE_{\alpha}^{*}|<\infty$ for $V\subseteq E^{0}$. In particular then $\alpha$ is summable if and only if $\beta$ is summable.
\end{lemma}

\begin{proof}
Assume that $|\alpha|=|\beta|=n$ and $\beta=\alpha_{k+1}\cdots \alpha_{n}\alpha_{1} \cdots \alpha_{k}$. If $\gamma \in VE_{\alpha}^{*}$ then $\gamma \alpha_{1}\cdots \alpha_{k} \in E^{*}r(\beta)$. If $\gamma \alpha_{1}\cdots \alpha_{k}=\delta \beta^{2}$ for some $\delta$ then we have that
$$
\gamma \alpha_{1}\cdots \alpha_{k}= \delta \alpha_{k+1}\cdots \alpha_{n}\alpha_{1} \cdots \alpha_{k} \alpha_{k+1}\cdots \alpha_{n}\alpha_{1} \cdots \alpha_{k} \Rightarrow
\gamma = (\delta \alpha_{k+1}\cdots \alpha_{n} )\alpha \; ,
$$
which contradicts that $\gamma \in E_{\alpha}^{*}$. Hence we have an injective map
$$
VE_{\alpha}^{*} \ni \gamma \mapsto  \gamma \alpha_{1}\cdots \alpha_{k} \in E_{\beta}^{*} \sqcup \{ \delta \beta \; | \; \delta \in E_{\beta}^{*} \} \; ,
$$
and hence $|VE_{\alpha}^{*}|<\infty$ if $|VE_{\beta}^{*}|<\infty$. By symmetry it follows that $\alpha$ satisfies $(b')$ if and only if $\beta$ does. If $\alpha=\gamma^{n}$ with $n\geq 2$ and $\sigma^{k}(\alpha^{\infty} )=\beta^{\infty}$ then 
$$
\sigma^{|\gamma|}(\beta^{\infty}) = \sigma^{|\gamma|}(\sigma^{k}(\alpha^{\infty} ))=\sigma^{k}(\alpha^{\infty} ) \; ,
$$
which implies that $\beta$ also does not satisfy $(a)$ in Definition \ref{def217a}. This proves the Lemma.
\end{proof}

To analyse the harmonic measures we will need to work with the \emph{eventually cyclic paths}, which are paths on the form $\beta \alpha^{\infty}\in E^{\infty}$ where $\alpha$ is a loop and $\beta \in E^{*}$ satisfies that $s(\beta)=r(\alpha)$. We set $\mathcal{C}$ to be the set of eventually cyclic paths, i.e.
$$
\mathcal{C} = \left \{ x\in \partial E \; | \; \exists \alpha, \beta \in E^{*} \text{ with $\alpha$ a loop and $s(\beta)=r(\alpha)$ s.t. } x=\beta\alpha^{\infty} \right \} \; ,
$$
and we define for any subset $S \subseteq E^{0}$ the set $\mathcal{C}S$ as
\begin{align*}
\mathcal{C}S:&=  \left\{ x\in \mathcal{C} \; | \; s(x_{i}) \in S \text{ for infinitely many $i\in \mathbb{N}$}   \right\} \\
&= \left\{ x\in \mathcal{C} \; | \; x=\beta \alpha^{\infty} \text{ with $\alpha$ a loop, $\beta \in E^{*}r(\alpha)$ and $s(\alpha_{i}) \in S$ for some $1\leq i \leq |\alpha|$}   \right\} \; .
\end{align*}

\begin{lemma} \label{lemborelep}
If $S \subseteq E^{0}$ is a Borel set, then the set $\mathcal{C}S$ is an invariant Borel subset of $\partial E$. 
\end{lemma}

\begin{proof}
It is straightforward to check that $\mathcal{C}S$ is invariant under $\sigma$. Take a $n\in \mathbb{N}$. The map $r\circ \sigma^{i}: E^{\infty} \to E^{0}$ is continuous for any $i\in \mathbb{N}$, so the set
$$
M_{n}:=\{x\in E^{\infty} \; | \; \sigma^{n}(x)=x \} \cap \bigcup_{i=0}^{n-1} \{ x\in E^{\infty} \; | \;  r\circ \sigma^{i}(x) \in S \}
$$
is Borel in $E^{\infty}$ and hence in $\partial E$ by Lemma \ref{lemma11}. Since
$$
\mathcal{C} S = \bigcup_{n=1}^{\infty} \bigcup_{m=0}^{\infty}
\sigma^{-m}(M_{n}) \; 
$$
then $\mathcal{C}S$ is a Borel subset of $\partial E$, proving the Lemma.
\end{proof}

\section{Harmonic measures on eventually cyclic paths}\label{subsection23}
We will now turn to the analysis of harmonic measures.
We will not achieve a complete description of all the extremal harmonic measures as was the case for boundary measures, but we will describe all extremal harmonic measures concentrated on the eventually cyclic paths. It turns out that the existence of non gauge-invariant tracial weights is closely relate to the existence of such measures.

\begin{prop}\label{proptracenong}
Let $E$ be a second countable topological graph. Assume that $\nu$ is a non-zero extremal harmonic measure on $\partial E$ satisfying that $\nu(\mathcal{C}^{C})=0$. Then there exists a summable loop $\alpha$ in $E$ such that $\nu$ is concentrated on the countable set
$$
\left \{ \beta \alpha^{\infty} \; | \; \beta \in E_{\alpha}^{*}  \right \} \; .
$$
This implies that there exists a $r> 0$ such that
$$
\nu = r \cdot  \sum_{\beta \in E_{\alpha}^{*}  } \delta_{\beta \alpha^{\infty}}  \; .
$$ 
\end{prop}

\begin{proof}
Let $\{V_{i} \}_{i\in \mathbb{N}}$ be a countable basis of $E^{0}$ of open sets with compact closures. Then 
$$
\mathcal{C} = \bigcup_{i=1}^{\infty} \mathcal{C} V_{i} \; .
$$
Since $\nu$ is extremal and $\mathcal{C} V_{i} $ is an invariant Borel set by Lemma \ref{lemborelep}, there exists some $j$ with $\nu( (\mathcal{C} \overline{V_{j}})^{C} )=0$. Since $\overline{V_{j}}$ is a second countable compact Hausdorff space in the relative topology, we can view the topology on it as generated by a metric. We can therefore find a sequence $\{W_{i} \}$ of Borel sets in $\overline{V_{j}}$ such that $\nu((\mathcal{C} W_{i})^{C})=0$ for all $i$, such that $W_{i+1} \subseteq W_{i}$ for all $i$ and such that the diameter of $W_{i}$ is less than $1/i$. It follows that
$$
0= \nu\Big( \bigcup_{i\in \mathbb{N}} (\mathcal{C} W_{i})^{C}   \Big) 
= \nu \Big (  \Big( \bigcap_{i\in \mathbb{N}} \mathcal{C} W_{i} \Big)^{C} \Big)
= \nu \Big (  \Big( \mathcal{C} \bigcap_{i\in \mathbb{N}} W_{i} \Big)^{C} \Big)    \; ,
$$
but since $W_{i}$ has diameter less than $1/i$ and $\nu \neq 0$, this implies that $\bigcap_{i\in \mathbb{N}} W_{i}=\{v\}$ for some $v\in E^{0}$. We now have that
$$
0= \nu((\mathcal{C} \{v\} )^{C} ) \; .
$$
This implies in particular that $\mathcal{C} \{v\} \neq \emptyset$. Since $s$ is a local homeomorphism and $E$ is second countable, there exists an at most countable collection $\{\alpha^{i}\}_{i\in I}$ of loops satisfying that $s(\alpha^{i})=v$. It follows by definition of $\mathcal{C}$ that we can write
$$
\mathcal{C} \{v\} = \bigcup_{i\in I} \left \{   \beta (\alpha^{i})^{\infty} \; | \; \beta \in E^{*}v  \right \} \; ,
$$
and since each of these sets are invariant, there is a loop $\alpha$ such that $r(\alpha)=s(\alpha)=v$ and 
$$
\nu\left (  \{   \beta \alpha^{\infty} \; | \; \beta \in E^{*}r(\alpha)   \}^{C} \right) =0  \; .
$$
It follows that $\nu$ is concentrated on $\{ \beta \alpha^{\infty} \; | \; \beta \in E^{*}r(\alpha)   \}$. Let $n\in \mathbb{N}$ be the smallest number such that $\sigma^{n}(\alpha^{\infty})=\alpha^{\infty}$. By possible substituting $\alpha$ we can assume that $|\alpha|=n$. We claim that the surjective map
$$
E_{\alpha}^{*} \ni \beta \; \mapsto \; \beta \alpha^{\infty} \in \{   \beta \alpha^{\infty} \; | \; \beta \in E^{*}r(\alpha)   \}
$$ 
is also injective. If $\beta\alpha^{\infty}=\delta \alpha^{\infty}$ with $|\delta|<|\beta|$ then $\beta = \delta \beta '$, and hence using $\sigma^{|\delta|}$ on $\beta\alpha^{\infty}=\delta \alpha^{\infty}$ we get that $\beta'\alpha^{\infty}=\alpha^{\infty}$. Since $\beta'\in E_{\alpha}^{*}$ we can write $\beta'=\alpha^{k} \zeta$ with $0<|\zeta |< |\alpha|$, which implies that $\zeta \alpha^{\infty}=\alpha^{\infty}$. Iterating this equality we see that $\zeta^{\infty}=\alpha^{\infty}$, contradicting our choice of $\alpha$. Hence the map is surjective, and since $\nu(\{ \beta \alpha^{\infty}\})=\nu(\{ \alpha^{\infty}\})$ we get that
$$
\nu= \sum_{\beta \in E_{\alpha}^{*}} \delta_{\beta \alpha^{\infty}} \;  \nu(\{ \beta \alpha^{\infty}\})
= \nu(\{ \alpha^{\infty}\}) \cdot \sum_{\beta \in E_{\alpha}^{*}} \delta_{\beta \alpha^{\infty}}  \; .
$$
We now only have to prove that $\alpha$ is a summable loop. Fix therefore a $w\in E^{0}$. Since $\nu$ is regular on $\partial E$, there exists an open neighbourhood $V_{w}$ of $w$ in $E^{0}$ with $\overline{V_{w}}$ compact, and hence $\nu(Z(V_{w}))<\infty$. Since
$$
\nu(Z(V_{w})) =  r\cdot \sum_{\beta \in E_{\alpha}^{*}} \delta_{\beta \alpha^{\infty}}(Z(V_{w}))
= r \cdot | V_{w} E_{\alpha}^{*} | \; , 
$$
then $\alpha$ is in fact a summable loop.
\end{proof}

Similarly to the case of boundary measures, we can now prove that equivalence classes of summable loops in the graph are in bijective corresponding with the rays of extremal harmonic measures concentrated on eventually cyclic paths. 

\begin{thm}\label{thmcyclicbij}
Let $E$ be a second countable topological graph. There is a bijection between the rays of non-zero extremal harmonic measures concentrated on $\mathcal{C}$ and equivalence classes of summable loops $\alpha$ in $E$. The ray corresponding to such a loop $\alpha$ is on the form $\{ r \cdot \nu \; | \; r\geq0  \}$, where
\begin{equation} \label{eqcyclicmeasure}
\nu = \sum_{\beta \in E_{\alpha}^{*} } \delta_{\beta \alpha^{\infty}}  \; .
\end{equation}
\end{thm}

\begin{proof}
Let $\alpha \in E^{*}$ be a summable loop. The set $E_{\alpha}^{*} $ is countable by second countability of $E$. We define the measure $\nu$ as in \eqref{eqcyclicmeasure}. It follows as in the proof of Theorem \ref{thmsingbi} that $\nu$ is non-zero, regular, extremal, invariant and satisfies that $\nu(\mathcal{C}^{C})=0$. If two loops $\alpha$ and $\beta$ are equivalent, they clearly give rise to the same measure. Hence the map which sends an equivalence class of a summable loop $\alpha$ to a measure $\nu$ is well defined, and it is surjective by Proposition \ref{proptracenong}. To see that it is injective let $\nu_{i}$ be defined by the loop $\alpha^{i}$ for $i=1,2$ and assume $\nu_{1}=\nu_{2}$. Then $\nu_{2}((\alpha^{1})^{\infty})>0$, and hence $(\alpha^{1})^{\infty}=\beta (\alpha^{2})^{\infty}$ for some $\beta \in E_{\alpha_{2}}^{*}$. This implies that $\alpha_{1}$ and $\alpha_{2}$ are equivalent because $\sigma^{|\beta|}((\alpha^{1})^{\infty})=(\alpha^{2})^{\infty}$. Hence our map is a bijection which proves the Theorem.
\end{proof}

\begin{remark}
If $E^{0}$ is compact we can exchange the condition: For any $w\in E^{0}$ there exists an open neighbourhood $V_{w}$ of $w$ in $E^{0}$ satisfying $|V_{w}E_{\alpha}^{*}|< \infty$, with the condition that $|E_{\alpha}^{*}| < \infty$.

By Lemma \ref{lemmaequiv} then $|E_{\alpha}^{*}|<\infty$ if and only if $|E_{\beta}^{*}|<\infty$ for two equivalent loops $\alpha$ and $\beta$. It follows from Theorem \ref{thmcyclicbij} that the extremal rays of finite harmonic measures concentrated on $\mathcal{C}$ maps bijectively onto the equivalence classes of summable loops $\alpha$ satisfying the extra condition $|E_{\alpha}^{*}|<\infty$. This observation gives the Corollary below.
\end{remark}

\begin{cor} \label{core54}
Let $E$ be a second countable topological graph. There is a bijection between the set of extremal harmonic probability measures concentrated on $\mathcal{C}$ and the set equivalence classes of summable loops $\alpha$ with $|E_{\alpha}^{*}|< \infty$. The extremal harmonic probability measures corresponding to such a loop is on the form
$$
\nu = \frac{1}{|E_{\alpha}^{*}|}\sum_{\beta \in E_{\alpha}^{*}} \delta_{\beta \alpha^{\infty}} \; .
$$
\end{cor}

\section{Gauge-invariant tracial weights}\label{subsection24}

With the analysis of measures concentrated on eventually cyclic paths from Section \ref{subsection23} we can now give criteria for when there exists non gauge-invariant tracial weights.

\begin{prop} \label{propdecomp}
Assume that $E$ is a second countable topological graph. If there exists a tracial weight on $C^{*}(E)$ which is not gauge-invariant, then there exists a summable loop in $E$.

If there exists a tracial state on $C^{*}(E)$ which is not gauge-invariant, then there exists a summable loop $\alpha$ in $E$ with $|E_{\alpha}^{*}|<\infty$.
\end{prop}

\begin{proof}
Let $\omega$ be a tracial weight on $C^{*}(E)$ which is not gauge-invariant. Since $\partial E$ is second countable we can find a strictly positive function $h\in C_{0}(\partial E)$ such that 
$$
\omega(h) = \int_{\partial E} h \; \mathrm{d} \nu_{\omega} =1 \; . 
$$
Letting $\mathcal{W}(C^{*}(E))$ denote the tracial weights on $C^{*}(E)$, we make the following claim:

\bigskip

\emph{Claim: The set $\mathcal{N}:=\{\psi \in \mathcal{W}(C^{*}(E)) \; | \; \psi(h)\leq 1\}$ is compact.}

\bigskip

Let $\{E_{n}\}_{n\in \mathbb{N}} \subseteq C_{c}(\partial E)_{+}$ be an approximate identity as chosen in Section 4 in \cite{C}. To see that $\mathcal{N}$ is a closed set notice that
$$
\mathcal{N}= \bigcap_{n\in \mathbb{N}} \{\psi \in \mathcal{W}(C^{*}(E)) \; | \; \psi(E_{n}hE_{n})\leq 1\}  \; .
$$ 
Since $h$ is strictly positive there exists a $\varepsilon_{n}>0$ for each $n\in \mathbb{N}$ such that $h(x) \geq \varepsilon_{n}$ for $x\in \text{supp}(E_{n})$. If $\psi \in \mathcal{N}$ it follows that
$$
\psi(E_{n}^{2})\cdot \varepsilon_{n} = \psi(\varepsilon_{n} E_{n}^{2}) \leq \psi( h) =1 \; ,
$$
and hence $\psi(E_{n}^{2}) \leq \varepsilon_{n}^{-1}$. It follows that the elements of $\mathcal{N}$ restricted to $\mathcal{A}_{n}:=\overline{E_{n} C^{*}(E)E_{n} }$ are contained in the ball in $\mathcal{A}_{n}^{*}$ of radius $\varepsilon_{n}^{-1}$. The topological space we embed $W(C^{*}(E))$ in is $\varprojlim_{n} \mathcal{A}_{n}^{*}$, c.f. Section 4 of \cite{C}, and hence $\mathcal{N}$ is compact in this space. This proves the claim.

\bigskip

Let now $f\in C_{c}(\mathcal{G}_{E})$ be supported in $\Phi^{-1}(\{k\})$ for some $k\in \mathbb{Z} \setminus \{0\}$ satisfy that $\omega(f)\neq 0$. Since the $\psi \in \mathcal{N}$ with $\psi(f) \neq 0$ is an open set by definition of the topology on $\varprojlim_{n} \mathcal{A}_{n}^{*}$, we have by the Krein-Milman Theorem that there exists an element $\psi'  \in \mathcal{N}$ which is a finite convex combination of extremal elements in $\mathcal{N}$ with $\psi'(f) \neq 0$. In particular there exists an extremal element $\psi \in \mathcal{N}$ with $\psi(f)\neq 0$. It follows that $\psi\neq 0$, so since $h>0$ on $\partial E$ then $\psi(h)>0$. Since $\psi$ is extremal in $\mathcal{N}$ we must have $\psi(h)=1$.

To see that $\psi$ also has to be an extremal tracial weight on $C^{*}(E)$, write $\psi=\psi_{1}+\psi_{2}$ for two non-zero tracial weights $\psi_{1}$ and $\psi_{2}$ on $C^{*}(E)$. Since $\psi(h)=1$ then $\psi_{i}(h)<\infty$ for $i=1,2$. Since $h>0$ then $\psi_{i}(h)=0$ would imply that the associated measure $\nu_{\psi_{i}}=0$ which would contradict that $\psi_{i}\neq 0$. Hence we can write
$$
\psi=\psi_{1}(h)  \frac{\psi_{1}}{\psi_{1}(h)}+\psi_{2}(h)  \frac{\psi_{2}}{\psi_{2}(h)} \; ,
$$
which implies that $\psi_{1} =\psi \cdot \psi_{1}(h)$ since $\psi$ is extremal in $\mathcal{N}$. In conclusion, we have an extremal tracial weight $\psi$ with $\psi(f)>0$. Since $\text{supp}(f) \subseteq \Phi^{-1}(\{k\})$ then $\gamma_{z}(f)=z^{k} f$, so $\psi$ is not gauge-invariant.

By Theorem \ref{thmgroupoiddescrip} it follows that the subgroup $H\subseteq \mathbb{Z}$ with 
$$
\nu_{\psi}(\{ x \in \partial E\; | \; H=\text{Per}(x) \}^{C} )=0
$$
is not $\{0\}$. If $\text{Per}(x)\neq \{0\}$ then there exists $k>l$ such that $\sigma^{k}(x)=\sigma^{l}(x)$. Then $\sigma^{k-l}(\sigma^{l}(x)) = \sigma^{l}(x)$ and it follows that $x\in \mathcal{C}$. In conclusion $\nu_{\psi}(\mathcal{C}^{C})=0$, and by Proposition \ref{proptracenong} there exists a summable loop in $E$.

\bigskip

Assume now that $\omega$ is a tracial state and let $f\in C_{c}(\mathcal{G}_{E})$ be supported in $\Phi^{-1}(\{k\})$ for some $k\in \mathbb{Z} \setminus \{0\}$ satisfy that $\omega(f)\neq 0$. The set 
$$
M_{r} = \{ \psi \in C^{*}(E)^{*} \; | \; \lVert \psi \rVert\leq 1, \text{ $\psi$ is a trace and } |\psi(f)|\geq r \}
$$
is a weak$^{*}$ compact subset of the dual for all $r\geq 0$. Since $M_{r_{2}} \subseteq M_{r_{1}}$ for $r_{2} \geq r_{1}$, there exists a unique $r_{0}>0$ such that $M_{r_{0}}\neq \emptyset$ and $M_{r} = \emptyset$ for $r >r_{0}$. Assume that $\psi$ is an extremal element in $M_{r_{0}}$, and that $\psi=\psi_{1}'+\psi_{2}'$ for two non-zero tracial weights $\psi_{1}'$ and $\psi_{2}'$. Set $\lambda = \nu_{\psi_{1}'}(\partial E)$. Since $\nu_{\psi}(\partial E)=1$ defining $\psi_{1}=\lambda^{-1} \psi_{1}'$ and $\psi_{2}=(1-\lambda)^{-1} \psi_{2}'$ gives
$$
\psi= \lambda \psi_{1}+(1-\lambda) \psi_{2} \; .
$$
This implies that
$$
r_{0}=|\psi(f)|\leq \lambda |\psi_{1}(f) |+(1-\lambda)| \psi_{2}(f)| \leq \lambda r_{0}+(1-\lambda) r_{0} = r_{0} \; ,
$$
by definition of $r_{0}$. Hence we have equality throughout, and $\psi_{i} \in M_{r_{0}}$ for $i=1,2$. This proves that $\psi$ is a tracial state which is extreme in the set of tracial weights and which satisfies $\psi(f)\neq 0$. The rest of the proof follows as in the tracial weight case and by invoking Corollary \ref{core54}.
\end{proof}

\begin{prop} \label{proplast}
Assume that $E$ is a second countable topological graph. If there exists a summable loop $\alpha$ in $E$, then there exists a non gauge-invariant tracial weight on $C^{*}(E)$. If there exists a summable loop $\alpha$ in $E$ with $|E_{\alpha}^{*}|<\infty$ then there exists a non-gauge invariant tracial state on $C^{*}(E)$. 
\end{prop}

\begin{proof}
If there exists a summable loop $\alpha$ in $E$, then the measure $\nu=\sum_{\beta \in E_{\alpha}^{*}} \delta_{\beta \alpha^{\infty}}$ is an extremal harmonic measure by Theorem \ref{thmcyclicbij}. The unique subgroup $H$ corresponding to $\nu$ in Theorem \ref{thmgroupoiddescrip} is $|\alpha| \mathbb{Z}$. Let the state $\phi$ on $C^{*}(|\alpha| \mathbb{Z})$ satisfy that $\phi(u_{g})=1$ for all $g\in |\alpha| \mathbb{Z}$. Recall that the map $\Phi$ on $\mathcal{G}_{E}$ given by $(x, k, y)\mapsto k$ gives rise to the gauge-action $\{\gamma_{z}\}_{z\in \mathbb{T}}$ on $C^{*}(\mathcal{G}_{E})$. If $f\in C_{c}(\mathcal{G}_{E})$ is a positive function which is supported on a bisection inside $\Phi^{-1}(|\alpha|)$ and satisfies $f(\alpha^{\infty}, |\alpha|, \alpha^{\infty})=1$, then $\psi_{\nu, \phi}(f)>0$. Hence $\psi_{\nu, \phi}(\gamma_{z}(f))=z^{|\alpha|}\psi_{\nu, \phi}(f)$ and there exists a tracial weight which is not gauge-invariant. If $|E_{\alpha}^{*} |<\infty$ then $\nu$ is finite, and the weight $\psi_{\nu, \phi}$ can be normalized to a tracial state.
\end{proof}

Combining Proposition \ref{proplast} and Proposition \ref{propdecomp} we obtain the following two Theorems.

\begin{thm} \label{thmnongauge}
Let $E$ be a second countable topological graph. All tracial states on $C^{*}(E)$ are gauge-invariant if and only if there does not exists a summable loop $\alpha$ in $E$ with $|E_{\alpha}^{*} | < \infty$. 
\end{thm}

\begin{thm} \label{thmnongauge2}
Let $E$ be a second countable topological graph. All tracial weights on $C^{*}(E)$ are gauge-invariant if and only if there does not exists a summable loop $\alpha$ in $E$. 
\end{thm}

We will now use Proposition \ref{propdecomp} to provide an affirmative answer to a conjecture of Schaufhauser for second countable graphs. 

\begin{thm} \label{thmfreegraph}
If $E$ is a free second countable topological graph, then every tracial weight on $C^{*}(E)$ is gauge-invariant. This implies that there is a bijection between
\begin{enumerate}
\item \label{item1}the set of invariant measures on $\partial E$,
\item \label{item2}the set of vertex-invariant measures on $E^{0}$, and
\item \label{item3}the set of tracial weights on $C^{*}(E)$.
\end{enumerate}
The bijection between \eqref{item1} and \eqref{item2} is given by $\nu \mapsto r_{*} \nu$. The bijection between \eqref{item1} and \eqref{item3} is given by $\nu \mapsto \omega_{\nu}$ where $\omega_{\nu}$ is the weight
$$
\omega_{\nu}(a) = \int_{\partial E} P(a) \; \mathrm{d} \nu \quad \text{ for $a\in C^{*}(E)$.}
$$ 
\end{thm}

Since Theorem \ref{thmfreegraph} in particular implies that every tracial state on a second countable free topological graph is gauge-invariant, it gives an affirmative answer to Conjecture 6.5 in \cite{S} for second countable graphs $E$, i.e. for graphs $E$ with $C^{*}(E)$ a separable C$^{*}$-algebra. Before turning to the proof of the Theorem, let us write out the following Corollary, which follows from the observation that $E$ is a free topological graph if $C^{*}(E)$ is simple, c.f. Theorem 8.12 in \cite{K3}.

\begin{cor}
If $C^{*}(E)$ simple and separable, then every tracial weight on $C^{*}(E)$ is gauge-invariant. This implies that there is a bijection between
\begin{enumerate}
\item \label{item1}the set of invariant measures on $\partial E$,
\item \label{item2}the set of vertex-invariant measures on $E^{0}$, and
\item \label{item3}the set of tracial weights on $C^{*}(E)$.
\end{enumerate}
The bijection between \eqref{item1} and \eqref{item2} is given by $\nu \mapsto r_{*} \nu$. The bijection between \eqref{item1} and \eqref{item3} is given by $\nu \mapsto \omega_{\nu}$ where $\omega_{\nu}$ is the weight
$$
\omega_{\nu}(a) = \int_{\partial E} P(a) \; \mathrm{d} \nu \quad \text{ for $a\in C^{*}(E)$.}
$$
\end{cor}

\begin{proof} [Proof of Theorem \ref{thmfreegraph}]
We will prove that if there exists a non gauge-invariant tracial weight then the graph is not free. If there exists a non gauge-invariant tracial weight it follows from Proposition \ref{propdecomp} that there exists a summable loop $\alpha$ in $E^{*}$. We have already noted that summable loops are isolated, so by Lemma \ref{lemma44} the graph is not free.

\end{proof}


\bigskip

\end{document}